\documentclass[final]{siamltex}
\usepackage{amsmath, amsfonts, amssymb, epsfig}
\usepackage{color}
\usepackage{url}

\def\e{{\epsilon}}

\def\N{\mathcal{N}}
\def\ox{\overline{x}}
\def\ovr{\overline{\rho}}
\def\oA{\overline{A}}
\def\A{\mathcal{A}}
\def\B{\mathcal{B}}
\def\oa{\overline{a}}
\def\oG{\overline{G}}
\def\sjn{\sum_{j=1}^n}
\def\E{\mathcal{E}}
\def\oE{\mathcal{\overline{E}}}
\def\oN{\mathcal{\overline{N}}}
\def\0{{\bf 0}}
\def\1{{\bf 1}}
\def\e{{\bf e}}
\def\s{{\bf s}}
\def\R{\mathbb{R}}
\def\C{\mathbb{C}}

\providecommand{\com[2]}{\begin{tt}[#1: #2]\end{tt}}
\def\A{{\mathcal A}}
\def\G{{\mathcal G}}
\def\jsr{{\rm jsr}}

\newtheorem{remark}[theorem]{\indent {Remark}}

\title{Consensus with Ternary Messages}

\author{Alex Olshevsky\thanks{Department of Industrial and Enterprise Systems Engineering, University of Illinois at Urbana-Champaign, Urbana, IL, 61801, USA ({\tt aolshev2@illinois.edu}).  }}

\begin{document}

\maketitle

\begin{abstract} We provide a protocol for real-valued average consensus by  networks of agents which exchange only a single
message from the ternary alphabet $\{-1,0,1\}$ between neighbors at each step.  
Our protocol works on time-varying undirected graphs subject to a connectivity condition, has a worst-case convergence time which is polynomial in the number of agents and the initial values, and requires no global knowledge about the graph topologies on the part of each node to implement except for knowing an upper bound on the
degrees of its neighbors. 
\end{abstract}



\pagestyle{myheadings}
\thispagestyle{plain}

\section{Introduction\label{intro}}

\def\N{\mathcal{N}}
\def\ox{\overline{x}}
\def\ovr{\overline{\rho}}
\def\oA{\overline{A}}
\def\A{\mathcal{A}}
\def\B{\mathcal{B}}
\def\oa{\overline{a}}
\def\oG{\overline{G}}
\def\sjn{\sum_{j=1}^n}
\def\E{\mathcal{E}}
\def\oE{\mathcal{\overline{E}}}
\def\oN{\mathcal{\overline{N}}}
\def\0{{\bf 0}}
\def\1{{\bf 1}}
\def\e{{\bf e}}
\def\s{{\bf s}}
\def\R{\mathbb{R}}
\def\C{\mathbb{C}}
\def\A{{\mathcal A}}
\def\G{{\mathcal G}}
\def\jsr{{\rm jsr}}

The average consensus problem asks for a protocol by means of which $n$ agents with initial values $x_1(0), \ldots, x_n(0)$
can compute the average $\frac{1}{n} \sum_{i=1}^n x_i(0)$ subject to unpredictably time-varying restrictions on inter-agent communication. 
Recent years have seen a surge of interest in consensus protocols due to their widespread use as building blocks for distributed  control laws in multi-agent systems; for example, consensus protocols have been used for formation maintenance \cite{lby03, fm04, clvw05}, coverage control \cite{mkb04, mkb07, gcb08}, network clock synchronization \cite{lr06, ccsz11,sf11}, distributed
task assignment and partitioning \cite{cbh09}, statistical inference in sensor networks \cite{bgps06, xbk07, dsw08},  and many 
other contexts in which centralized control is absent and agent motion and time-varying interference can lead to repeated
failures of communication or sensing.

Consensus protocols typically involve each node $i$ maintaining a variable $x_i(t)$ which is updated from time $t-1$ to time $t$ by setting $x_i(t)$ to be a convex combination of those $x_j(t-1)$ for which $j$ and $i$ are neighbors in some undirected graph $G(t) = ( \{1, \ldots, n\}, E(t))$. The graph sequence $G(t)$ is meant to capture the constraints which dictate which pairs of agents can communicate or sense each other at each time.  It is assumed that each $G(t)$ has a self-loop at every node so that $x_i(t)$ can always depend on its previous value $x_i(t-1)$. A typical update is the Metropolis iteration (first introduced within the context of consensus in \cite{xbk07}), defined as
\begin{equation} \label{metropolis} x_i(t) = x_i(t-1) + \sum_{j \in N_i(t)} \frac{x_j(t-1) - x_i(t-1)}{D(i,j,t)} \end{equation} where $N_i(t)$ is the set of neigbors of node $i$ in $G(t)$, $d_i(t)$ is the degree of node $i$ in $G(t)$, and  $D(i,j,t)$ is a collection of numbers, defined for all pairs $i,j$ such that
$(i,j) \in G(t)$,  satisfying the symmetry conditions 
$D(i,j,t) = D(j,i,t)$ and the upper bounds $D(i,j,t) \geq \max ( d(i), d_j(t) )$. Note that choosing $D(i,j,t) = \max( d_i(t), d_j(t) )$ works, and choosing 
$D(i,j,t)=n$ for all $(i,j) \in G(t)$ works too, as well as any upper bound on $n$ known to the nodes. We will assume that the numbers $D(i,j,t)$ are uniformly bounded from above. We note that the nodes do not need to know any global information about the sequence $G(t)$ to implement this iteration, and we may in fact assume that $G(t)$ is exogenously or adversarially given subject to the connectivity constraint we will next
describe. Namely, we will assume that the graph sequence $G(t)$ satisfies the long-term connectivity condition  \begin{equation} \label{connectivity} \mbox{ For each } k, \mbox{ the graph } \left(\{1,\ldots,n\}, \cup_{t=k}^{\infty} E(t) \right) \mbox{ is connected}, \end{equation} meaning that the communication restrictions faced by the nodes do not
disconnect the network into noncommunicating groups after some finite time $k$. Under this assumption, 
the Metropolis iteration of Eq. (\ref{metropolis}) has the property  that \[ \lim_{t \rightarrow \infty} x_i(t) = \frac{1}{n} \sum_{i=1}^n x_i(t) ~~\mbox{ for each } i = 1, \ldots, n\] meaning that all agents suceed in converging to the average. We refer the
reader to the papers \cite{bhot05, JLM03, M05, OSM04, TBA86, xbk07} for  proofs of this and similar assertions.  

Consensus protocols have proven to be useful for multi-agent control due to their attractive robustness properties (namely, that the communication sequence $G(t)$ can vary arbitrarily and unpredictably subject only to the relatively weak long-term connectivity constraint of Eq. (\ref{connectivity})) and most importantly due to their local, distributed nature. Indeed, all the information that node $i$ needs to implement the Metropolis iteration, from its neighbor's values to upper bounds on its neighbors' degrees, is locally available. In the event that information exchange is wireless, nodes may broadcast their values and degrees to their neighbors; on the 
other hand, if $x_i(t)$ are quantities like positions or velocities which are physically measured by neighbors, implementation of the Metropolis iteration usually requires nodes to know upper bounds $D(i,j,t)$ on the degrees of their neighbors. These bounds can be based on the inherent sensing constraints (e.g., no node can be sensing too many other nodes simultaneously) or on an upper bound for the total number of nodes in the system.

We remark that a number of recent advances in multi-agent control have proceeded by reduction to an appropriately defined consensus problem \cite{gcb08, fm04, clvw05, lr06, sf11}. Moreover, under slightly stronger assumptions on the sequence $G(t)$, it is possible to design average consensus algorithms with convergence time bounds which scale quadratically with the number of nodes $n$ \cite{noot09}. 

However, a limitation of consensus protocols lies in the assumption that agent $i$ can update $x_i(t)$ as a function of the states $x_j(t-1)$ of its neighbors in the graph $G(t)$. While this is very natural in some settings (such as formation or coverage control, where $x_j(t)$ represents the position of agent $j$, which neighboring agent $i$ can be reasonably assumed to sense) in other contexts (such as clock synchronization or statistical inference in sensor networks) this implies that the agents can exchange real numbers at every step, a clearly unrealistic assumption. 

Indeed, sensor networks which run  distributed estimation protocols cannot transmit real-valued messages to neighbors; the physics of information transmission over the wireless medium instead limit the nodes to messages which come from a finite alphabet. The goal of this paper is to provide 
a consensus protocol which works in this setting. Specifically, we provide a deterministic protocol for the nodes to exchange messages with neighbors and update $x_i(t)$ which has the following properties:

\smallskip 

\begin{enumerate} \item  For every time-varying sequence $G(t)$ subject to a certain connectivity assumption, each $x_i(t)$ converges to the average of the initial values $\frac{1}{n} \sum_{i=1}^n x_i(0)$.
\item The number of bits that is transmitted from agent to agent at each time $t$ is bounded above independently of $t,n$ and all
other problem parameters (in fact
in our protocol the agents transmit a single element from the ternary alphabet $\{-1,0,1\}$ to each
neighbor at every $t$). 
\item The protocol does not require any global knowledge of the graph sequence $G(t)$ on the part of the nodes. 
\item The protocol has a worst-case convergence time which is provably scales polynomially in the number of nodes 
$n$ and the initial values $x_i(0)$. 
\end{enumerate} 

\smallskip

 We remark that features (1) and (3) are particularly crucial since much of the
attractiveness of consensus protocols arises from their abilities to cope with link failures and their local,
distributed nature. Feature (4) is clearly useful: the appeal of any protocol 
is increased by the availability of polynomial worst-case convergence bounds. Finally, feature (2) ensures that we do not implicitly assume that an arbitrarily large amount of data can transmitted on each link before the graph changes

Much like the Metropolis protocol, our protocol assumes every node has access to upper bounds $D(i,j,t)$ on the degres of its neighbors. Moreover, we remark that
our protocol may be considered to have binary (rather than ternary) messages if we identify no transmission on a link with a ``$0$'' message.

A number of consensus protocols in which agents exchange finitely many bits at every step have recently been proposed. We mention the paper \cite{kbs07} which
initiated the literature, the follow-on works \cite{noot09, frasca, xie2, princeton, lm12}, as well as the related papers \cite{cai1, cai2}. In \cite{kbs07} and follow-up papers nodes exchange quantized values in order to eventually approach a neighborhood of the initial average. Variations on this theme have been 
studied in the past several years, and we mention that the recent preprint \cite{princeton} has the 
best known convergence time bounds at the time of writing for these updates in discrete time. Finally, a similar consensus process has been studied recently over finite fields in \cite{finite}.

More directly related to the current work are the papers \cite{cbz10, cffz10, fcfz09, lfxz11, xiesiam} which considered the problem using quantized communications to converge to the real-valued average of the initial values {\em exactly}, as we do here. The present paper is directly motivated by this literature, in particular by the observation that each of the protocols in these papers lacks at least two of the features (1)-(4) above. Indeed, many of these protocols only work on fixed graphs, and furthermore utilize updates at each node which depend on global information about the network (for example by having each node's update depend on the eigenvalues of a matrix built from the network). By contrast, the protocol we propose in this paper requires no knowledge about the time-varying
graph sequence $G(t)$ to implement with the possible exception of each node having  
upper bounds on the degrees of its neighbors.  

\smallskip

We now describe the organization of the remainder of the paper. We begin by stating our main result in Section \ref{statement}, where we provide an informal description as well as a formal statement of our protocol. The proof of the main convergence theorem is given in the subsequent Section \ref{mainthmproof}. Some simulations of our
our protocol are provided in Section \ref{simulations}. Finally, Section \ref{conclusions} provides some conclusions and lists some open problems. We note in particular that our new consensus 
protocol makes more stringent than usual assumptions on the graph sequence $G(t)$ and has more storage than the ordinary consensus 
algorithms, and we pose the improvement of these features as open problems. 

\section{Our results\label{statement}} We first reprise the notation we have introduced:  $G(t)$ is a sequence of undirected graphs with a self-loop at every node, $N_i(t)$ is the set of neighbors 
of node $i$ in $G(t)$, and $d_i(t)$ is the degree of node $i$ in $G(t)$. 

\subsection{Intuitive description of the protocol} We begin by informally describing the idea behind our protocol. We would like to run the Metropolis update of Eq. (\ref{metropolis}), but without
the ability to transmit real numbers, node $i$ will not know the values of its neighbors $x_j(t-1)$ exactly. Consequently, node $i$ will maintain an estimate of $x_j(t)$; we will use the notation $\widehat{x}_{i,j,{\rm in}}(t)$ for the estimate
that node $i$ has at time $t$ for $x_j(t)$.  


At each time $t$, node $i$ will receive a message from each of its neighbors $j \in N_i(t)$ from the alphabet $\{-1,0,1\}$. If it receives a $+1$, it will add $1/t^{\alpha}$  to $\widehat{x}_{i,j, {\rm in}}$; if it receives a $-1$, it adds a $-1/t^{\alpha}$ to 
the same; and if 
it receives a zero, it leaves $\widehat{x}_{i,j, {\rm in}}$ unchanged. Naturally, node $j$ decides to send a $1,0,-1$ depending on whether $\widehat{x}_{i,j, {\rm in}}(t-1)$ is too low by at least $1/t^{\alpha}$, within $1/t^{\alpha}$ of the true value, or too high by a factor of at least $1/t^{\alpha}$, respectively. A key point is that even though it is node $i$ which maintains the variable $\widehat{x}_{i,j, {\rm in}}$, node $j$ knows 
what it is because it is built from previous messages sent by node $j$ to node $i$. 

Note that this update rules may require nodes to send different messages to different neighbors. 

The number $\alpha$ will be in  $(0,1)$ so that
\[ \sum_{t=1}^n \frac{1}{t^{\alpha}} = + \infty \] which will be key in ensuring that the estimates $\widehat{x}_{i,j, {\rm in}}$ increasingly 
become accurate for links that appear regularly. 

Meanwhile, the nodes will implement the Metropolis update, however with each node $i$ using its latest estimate $\widehat{x}_{i,j, {\rm in}}$ 
instead of the true value $x_j(t-1)$. However, node $i$ will not include all its neighbors in the Metropolis updates; 
rather, it will include only those nodes $j$ which sent it a zero at time $t$ (meaning that node $i$ estimate of $j$'s value is 
not too far from the truth) and whose associated estimates $\widehat{x}_{i,j, {\rm in}}(t)$ are sufficiently far from $x_i(t)$ (so that the inevitable 
pertubration error inherent in using imprecise estimates does not affect the convergence analysis too much).

Furthermore,  we introduce a stepsize of $1/t^{\beta}$ 
into the Metropolis algorithm where $\beta>\alpha$ and $\beta \in (0,1)$. By choosing $\beta > \alpha$, we 
ensure that agents change their values $x_i$ slower than estimates $\widehat{x}_{i,j, {\rm in}}$ change.  Intuitively, while 
the estimates $\widehat{x}_{i,j, {\rm in}}$ will get accurate by 
an additive factor of $1/t^{\alpha}$ whenever they are very inaccurate, the values $x_i$ will have their movement
attenuated by a factor of $1/t^{\beta}$.  The introduction of the $1/t^{\beta}$ stepsize ensures that not only do the estimates
eventually ``catch up'' to the true values but also that we can give simple bounds on
how long it takes until estimates become accurate. 

\subsection{Formal description of the protocol\label{description}}

The nodes begin with initial values $x_i(0)$. Every node $i$ will maintain variables $\widehat{x}_{i,j, {\rm in}}(t),\widehat{x}_{i, j, {\rm out}}(t)$ for every node which has been its neighbor 
in some past $G(t)$. Node $i$ initializes $$\widehat{x}_{i,j, {\rm in}}(t-1)=\widehat{x}_{i, j, {\rm out}}(t-1)=0$$ at the
first time $t$ when the edge $(i,j)$ belongs to $E(t)$.  

At each iteration $t=1, 2,3, \ldots$, node $i$ will send a value from the set $\{-1,0,1\}$ to each of its neighbors $j$ in $G(t)$. The value $i$ sends to $j$ is  \[ q_{ i \rightarrow j} (t) = R \left[ t^{\alpha} ({x_i(t-1) - \widehat{x}_{i,j,{\rm out}}(t-1)}) \right] \] where \[ R[x] = \begin{cases} 1 & \mbox{ if } x > 1 \\ 0 & \mbox{ if } -1 \leq x \leq 1 \\ -1 & \mbox{ if } x < -1 \end{cases} \] Note that node $i$ may send  different messages to different neighbors.  After sending $q_{i \rightarrow j}(t)$ to each neighbor $j \in N_i(t)$ and receiving $q_{j \rightarrow i}(t)$ from the same, node $i$ updates the values $\widehat{x}_{i,j, {\rm in}}, \widehat{x}_{i,j, {\rm out}}$ as
\begin{eqnarray*} \widehat{x}_{i,j, {\rm out}}(t) &  = &   \widehat{x}_{i,j, {\rm out}}(t-1) + \frac{q_{i \rightarrow j}(t)}{t^{\alpha}} \\
\widehat{x}_{i, j, {\rm in}}(t) &  = &   \widehat{x}_{i,j, {\rm in}}(t-1) + \frac{q_{j \rightarrow i}(t)}{t^{\alpha}}
 \end{eqnarray*} If $(i,j) \notin G(t)$ and consequently no message from $j$ to $i$ was sent, but $\widehat{x}_{i,j, {\rm in}}, \widehat{x}_{i,j, {\rm out}}$ have  been previously initialized, then node $i$ simply keeps the variables $\widehat{x}_{i,j, {\rm in}}, \widehat{x}_{i,j, {\rm out}}$ as they are: 
\begin{eqnarray*} \widehat{x}_{i, j, {\rm in}}(t) & = & \widehat{x}_{i,j, {\rm in}}(t-1) \\ 
 \widehat{x}_{i,j, {\rm out}}(t) & = & \widehat{x}_{i,j, {\rm out}}(t-1). 
\end{eqnarray*}   Each node then updates its value $x_i(t)$ as
 \begin{equation} \label{mainiter2} x_i(t) =    x_i(t-1) + \frac{1}{t^{\beta}} \sum_{j \in S(i,t)} \frac{\widehat{x}_{i,j,{\rm in}}(t) - \widehat{x}_{i,j, {\rm out}}(t)} {4 D(i,j,t)} \end{equation}
 Here $D(i,j,t)$ is, as in the Metropolis protocol, a collection of numbers satisfying the symmetry conditions $D(i,j,t) = D(j,i,t)$ and the upper bounds $D(i,j,t) \geq 
\max( d_i(t), d_j(t))$; and $S(i,t)$ is the set of neighbors $j$ of node $i$ in $G(t)$ which satisfy $$|\widehat{x}_{i,j,{\rm in}}(t)-\widehat{x}_{i,j, {\rm out}}(t)| > \frac{4}{t^{\alpha}}$$ and $$q_{j \rightarrow i}(t)= 0, ~~~~q_{i \rightarrow j}(t) = 0.$$ 
 
\subsection{Main result} We now provide the main convergence theorem for our protocol. We begin with a few definitions needed to specify the class of graph sequences on which our protocol is guaranteed to work.

\bigskip
%

\begin{definition} We will call the graph sequence $G(t)$ $B$-core-connected if there exists a set of edges $E_{\infty} \subset \cup_t E(t)$ such that the graph $(\{1, \ldots, n\}, E_{\infty})$ is 
connected; and \[ E_{\infty} \subset \cup_{t=kB+1}^{(k+1) B} E(t)  \] for every nonnegative integer $k$. 
\end{definition}

\bigskip

That is, a sequence is $B$-core-connected if there is a set of edges, forming a connected graph, each of which
appears in every interval $[kB+1, (k+1)B]$. We will say that the edges in $E_{\infty}$ are {\em core edges}. 

Our main result provides upper bounds on the time until our consensus protocol reduces a certain measure of disagreement to a small value
forever. We now define this measure (this is $V_2(t))$ in the next definition) as well as several related
concepts. 

\bigskip

\begin{definition} \begin{eqnarray*} 
M(x) & =  & \max_{i=1, \ldots, n} x_i \\
m(x) & = & \min_{i=1, \ldots, n} x_i \\
W(x) & = & M(x) - m(x) \\
V_2(x) & = & \sqrt{\sum_{i=1}^n \left( x_i - \frac{1}{n} \sum_{j=1}^n x_j \right)^2} \\ 
\end{eqnarray*} 
\end{definition} Note that will use the natural shorthands $M(t), m(t), W(t),  V_2(t)$  for $M(x(t)), m(x(t)), 
W(x(t)), V_2(x(t))$. Finally, we will use the notation $D = \sup_{i,j,t} D(i,j,t)$ for the supremum of the all the 
quantities $D(i,j,t)$ upper bounding the degrees. Note that if $D(i,j,t) = \max_i (d_i(t), d_j(t))$ then we 
can trivially bound $D \leq n$. 

\bigskip

We can now state the main result of this paper. 

\bigskip

\begin{theorem} \label{precursorthm} If $0 < \alpha < \beta < 1$ then for all nodes $i$, initial values $x(0)$, and $B$-core-connected sequences $G(t)$, it is true that \[ \lim_{t \rightarrow \infty} x_i(t) = \frac{1}{n} \sum_{j=1}^n x_j(0). \] Moreover, if\footnote{The notation $\lceil x \rceil$ refers to the smallest integer which is at least $x$.} \begin{small} \begin{eqnarray*} t \geq
2^{\frac{1}{1-\beta}} \left( 2^{\frac{2}{1-\alpha}} \lceil 32B + 8 B W(0) \rceil^{\frac{1}{\beta-\alpha}} + \left( 32 B ||x(0)||_{\infty} \right)^{\frac{2}{1-\alpha}} + 11B+ (300 n^3 D B )^{\frac{1}{1-\beta}} 
\right) && \\ + \max \left( \left( 150 n^3 D B \log \frac{V_2(0)}{\epsilon} \right)^{\frac{1}{1-\beta}}, \left( \frac{8 n^{1.5}}{\epsilon} \right)^{\frac{1}{\alpha}} \right) && \end{eqnarray*} \end{small} then \[ V_2(x(t)) \leq \epsilon \] for all $i,j$. 
\end{theorem}   

\bigskip

Our result states that the consensus protocol of the previous section succeeds in computing the average on 
$B$-core-connected sequence and provides a bound for its convergence time. 

\bigskip

\begin{remark} While the above convergence time expression is somewhat unwieldy, we emphasize that it is polynomial
in $n, \epsilon, B,  W(0), V(0), ||x(0)||_{\infty}$ for any choice of $\alpha$ and $\beta$ satisfying the assumption $0 < \alpha < \beta < 1$. 
\end{remark}

\bigskip

\begin{remark} Observe that there is a tradeoff between the steady-state and transient terms in the convergence time bound. 
In particular, as $\epsilon \rightarrow 0$, the term $(8 n^{1.5}/\epsilon)^{1/\alpha}$ is going to dominate all the other
terms in the convergence time expression. Consequently, 
to obtain the best asymptotic decay rate we should choose $\alpha$ close to $1$ which means we must also choose $\beta$ close
to $1$.  In that case, error decay at time $t$  decays nearly as well as $O(n^{1.5}/t)$ as $t$ approaches
infinity. However, choosing $\alpha_n \rightarrow 1$ and $\beta_n \rightarrow 1$ causes the transient term to blow up, so that it 
takes longer and longer until the asymptotically dominant term dominates the other terms. 

In general, there is no single best choice of $\alpha, \beta$ which optimizes the convergence time expression; rather every choice gives us a tradeoff between 
transient bounds and steady-state decay. For example, choosing $\alpha=3/4, \beta=7/8$ gives us that the time until 
$V_2(x(t)) \leq \epsilon$ is \[ O \left( \left( B + B W(0) \right)^8 \right) + O \left( B ||x(0)||_{\infty} \right)^{8} + O(B) +  O(n^{24} D^8 B^8) + \max \left( O(n^{24} D^8 B^8 (\log \frac{V_2(0)}{\epsilon})^8, O \left( \frac{ n^2}{\epsilon^{4/3}} \right) \right)\] Note that every term which does not have an $\epsilon$ in it will become negligible as $\epsilon \rightarrow 0$. In this limit, the dominant term will be
$O( \frac{8 n^2}{\epsilon^{4/3}})$ which will grow faster as $\epsilon \rightarrow 0$ than the logarithmic term
$O( n^{24} D^8 B^8 \log \frac{V_2(0)}{\epsilon})$.   

On the other hand, choosing $\alpha=1/4, \beta=1/2$ gives us that the time until $V_2(x(t)) \leq 
\epsilon$ is \[  O \left( \left( B + B W(0) \right)^4 \right) + O \left( B ||x(0)||_{\infty} \right)^{8/3} + O(B) + O(n^{6} D^2 B^2) + \max \left( O(n^{6} D^2 B^2 (\log \frac{V_2(0)}{\epsilon})^2, O\left( \frac{n^6}{\epsilon^{4}} \right) \right) \] which has a worse asymptotic term $O\left( \frac{n^6}{\epsilon^{4}} \right)$ but the rest of the terms are smaller if $B, W(0), ||x(0)||_{\infty}, V_2(0)$ are large.
\end{remark}

\bigskip

\begin{remark} The analysis of consensus algorithms usually relies on a weaker notion of connectivity, namely
the so-called $B$-connectivity (sometimes called uniform connectivity) condition:  a sequence $G(t)$ of undirected graphs is 
called $B$-connected if for each $k$, the graph
\[ \left(  \{1, \ldots, n\}, \cup_{t=kB+1}^{(k+1) B} E(t) \right) \] is connected. For any fixed $B$, there are more 
$B$-connected sequences than $B$-core-connected sequences. It is an open question to prove a convergence time 
bound similar to our main result on $B$-connected sequences. \end{remark}

\bigskip

\begin{remark} Our protocol requires node $i$ to keep track of the numbers $\hat{x}_{i,j, {\rm in}}, \hat{x}_{i,j, {\rm out}}$ for every other node $j$ it interacts with. By contrast, the standard consensus algorithm keeps track of only
a single real number, namely $x_i(t)$ at node $i$. Unfortunately, it seems that a ``storage blowup'' phenomenon of this sort is unavoidable, though it 
remains an open question to prove this in any formal sense. 

As a consequence, our consensus protocol is is most attractive on graphs sequences $G(t)$ in which every node interacts with
a relatively small number of neighbors. One such example is the geometric random graph model of wireless networks, in which sensors have locations in $[0,1]^2$ and every node is connected to an expected $O(\log n)$ neighbors \cite{penrose}. Our consensus protocol 
will work in such networks, even if unpredictable interference or maliicous jamming makes links unreliable, with every node needing to store only $O(\log n)$ additional estimates. \end{remark} 

\bigskip

\begin{remark} We have not assumed that the nodes know the constant $B$. However, if the nodes do happen to know 
$B$ or an upper bound on it, they may use it to somewhat reduce their storage requirements. Rather than keep track of the estimates
$\widehat{x}_{i,j, {\rm in}}$ for every node $j$ that it has interacted with in the past, node $i$ can instead just keep track of estimates for 
just those nodes it has interacted with in the past $B$ steps. 

Our main result, Theorem \ref{precursorthm}, will then hold verbatim, since our proof of this theorem (in the next section) proceeds by analyzing reductions in $V_2(t)$ 
obtained when neighbors connected by core edges include each other in their Metropolis updates, which are unaffected by this modification. 
\end{remark}

\bigskip

\begin{remark} The introduction of the stepsize $1/t^{\beta}$ with $\beta > \alpha$ is crucial for our proof here. A main thrust of our argument 
is that after some transient period, we will have that $1/t^{\beta}$ is negligible compared to $1/t^{\alpha}$, so that estimates across each core link
get accurate due to messages between nodes much faster than they drift apart due to node updates. This argument is the source of the exponent $\frac{1}{\beta-\alpha}$ in the first term of our final convergence time bound. 
\end{remark}

\section{Main proof\label{mainthmproof}} This section is devoted to the proof of our main result, 
Theorem \ref{precursorthm}. We will begin with a long sequence of lemmas which will establish properties of our 
consensus protocol, putting off the calculations that lead to the bound of Theorem \ref{precursorthm} as much as possible. 

Roughly speaking, our initial sequence of lemmas will establish elementary properties of the protocol. In particular, we rewrite 
our protocol in more convenient form and establish that several Lyapunov functions are conserved by our protocol. Our first substantive statement
then comes in Lemma \ref{correctestimate} which states that, across core edges, estimates get more and more accurate over time. Shortly
thereafter, Lemma \ref{mainlemma} establishes that, after a transient period, the dispersion of the values $x_i(0)$ will decrease by
a certain proportion whenever it is not already small. The complete proof of Theorem \ref{precursorthm} follows shortly afterwards.

\bigskip

We now proceed to our first
lemma, which states the natural fact that the estimate $\widehat{x}_{j,i, {\rm out}}$ maintained by node $j$ equals 
the estimate $\widehat{x}_{i,j, {\rm in}}$ maintained by node $i$ (intuitively, this says that each node knows all other node's estimates
of its own value).

\smallskip 

\begin{lemma} \label{estimates} \[ \widehat{x}_{i,j, {\rm in}}(t) = \widehat{x}_{j,i, {\rm out}}(t) \] whenever 
these variables are defined. 
\end{lemma} 

\bigskip

\begin{remark} Recall that the variables $\widehat{x}_{i,j, {\rm in}}$ and $\widehat{x}_{j,i, {\rm out}}$ are 
initialized and maintained at the first time $(i,j)$ appears in the sequence $G(t)$. 
\end{remark}

\bigskip

\begin{proof} The first time $(i,j) \in E(t)$, the variables $\widehat{x}_{i,j, {\rm in}}(t-1)$ and $\widehat{x}_{j,i, {\rm out}}(t-1)$ are initialized to zero by both nodes $i$ and $j$. Subsequently, they are updated as (see Section \ref{description}) 
\begin{eqnarray*} \widehat{x}_{j,i, {\rm out}}(t) &  = &   \widehat{x}_{j,i, {\rm out}}(t-1) + \frac{q_{j \rightarrow i}(t)}{t^{\alpha}} \\
\widehat{x}_{i, j, {\rm in}}(t) &  = &   \widehat{x}_{i,j, {\rm in}}(t-1) + \frac{q_{j \rightarrow i}(t)}{t^{\alpha}}
 \end{eqnarray*} at each time $t$ such that $(i,j) \in E(t)$. Consequently, they have the same value at all times.
\end{proof}

\bigskip

The next corollary tells us that if node $i$ includes its estimate of the value of node $j$ in its Metropolis update,  then, 
symmetrically, node
 $j$ includes its estimate of node $i$'s value in its own Metropolis update. 

\smallskip

\bigskip

\begin{corollary} \label{symlinks} $j \in S(i,t)$ if and only if $i \in S(j,t)$. 
\end{corollary}

\bigskip

\begin{proof} Suppose $j \in S(i,t)$. This happens if and only if $q_{i \rightarrow j}(t)= q_{j \rightarrow i}(t) = 0$ and 
\begin{equation} \label{bigdifference} | \widehat{x}_{i,j, {\rm in}}(t) - \widehat{x}_{i,j, {\rm out}}(t)| > \frac{4}{t^{\alpha}}. \end{equation} The condition
$q_{i \rightarrow j}(t)= q_{j \rightarrow i}(t) = 0$ is naturally symmetric between $i$ and $j$. Moreover, using Lemma \ref{estimates}, we 
have that Eq. (\ref{bigdifference}) is equivalent to 
\[ | \widehat{x}_{j,i, {\rm out}}(t) - \widehat{x}_{j,i, {\rm in}}(t) | > \frac{4}{t^{\alpha}}, \] so that
$i \in S(j,t)$. 
\end{proof} 

\bigskip


We next observe that the update equations of our protocols may be rewritten in a convenient form:

\bigskip
 
\begin{lemma} \label{wbounds} Defining \[ w_{ij}(t-1) = \begin{cases}   0 & \mbox{ if } j \notin S(i,t) \\ \frac{\widehat{x}_{i,j, {\rm in}}(t) - \widehat{x}_{i,j, {\rm out}}(t)}{x_{j}(t-1) - x_i(t-1) }   & \mbox{
 else } \end{cases} \]  we then have that  
\begin{equation} \label{update} x_i(t) = x_i(t-1) + \frac{1}{t^{\beta}} \sum_{j \in S(i,t)} \frac{w_{ij}(t-1)}{ 4 D(i,j,t) } (x_j(t-1) - x_i(t-1)) \end{equation} and moreover \[ \frac{2}{3} \leq w_{ij}(t-1) \leq 2 \] for all $i,j$ such that $j \in S(i,t)$. 
\end{lemma}

\bigskip

\begin{proof} The only thing that needs to be proven are the lower and upper bounds on $w_{ij}(t-1)$.  Now  if $j \in S(i,t)$  we have \begin{equation} \label{atleast4} |\widehat{x}_{i,j, {\rm in}}(t) - \widehat{x}_{i,j, {\rm out}}(t)| >  \frac{4}{t^{\alpha}} \end{equation} Moreover, we also have that $q_{i \rightarrow j}(t) =0$, which means that
\[ | \widehat{x}_{i,j, {\rm out}}(t) - x_i(t-1) | = | \widehat{x}_{i,j, {\rm out}}(t-1) - x_i(t-1) | \leq \frac{1}{t^{\alpha}} \] and similarly, $q_{j \rightarrow i}(t)=0$ which means \begin{eqnarray*}  |\widehat{x}_{i,j, {\rm in}}(t) - x_j(t-1)|  =
|\widehat{x}_{i,j, {\rm in}}(t-1) - x_j(t-1)| =  |\widehat{x}_{j,i, {\rm out}}(t-1) - x_j(t-1)|  & \leq &  \frac{1}{t^{\alpha}}.
\end{eqnarray*} where the second equality used Lemma \ref{estimates}. We therefore have 
\[ \left| ~ | \widehat{x}_{i,j, {\rm in}}(t) - \widehat{x}_{i,j, {\rm out}}(t) | - |x_j(t-1) - x_i(t-1)| ~ \right| \leq \frac{2}{t^{\alpha}} \] and
combining this with Eq. (\ref{atleast4}), 
\[ w_{ij}(t-1) = \frac{\widehat{x}_{i,j, {\rm in}}(t) - \widehat{x}_{i,j, {\rm out}}(t)}{x_{j}(t-1) - x_i(t-1) }   \in [ \frac{4}{6}, \frac{4}{2} ], \] which completes the proof.
\end{proof} 

\bigskip

\begin{remark} As a consequence of  Lemma \ref{estimates}, we have that $w_{ij}(t) = w_{ji}(t)$ for all $i,j,t$. \label{symw} 
\end{remark} 
 
\bigskip Our next lemma rewrites the update dynamics of our consensus protocol in matrix form.
\bigskip
 
\begin{lemma} We may rewrite the main update Eq. (\ref{mainiter2}) as  \label{convexrewriting} \[ x(t) = (1 - \frac{1}{t^{\beta}}) x(t-1) + \frac{1}{t^{\beta}} A(t-1) x(t-1), \] where $A(t-1)$ is a symmetric, stochastic, diagonally dominant matrix with the property that if $a_{ij}(t-1)$ is positive, then $a_{ij}(t-1) \geq 1/(8 D(i,j,t))$.  \label{symstoch}
\end{lemma} 

\bigskip

\begin{proof} From Eq. (\ref{update}), \begin{footnotesize}
\begin{eqnarray*} x_i(t) & = & \left( 1 - \frac{1}{t^{\beta}} \sum_{j \in S(i,t)} \frac{w_{ij}(t-1)}{ 4 D(i,j,t) } \right) x_i(t-1) +  \frac{1}{t^{\beta}}  \sum_{j \in S(i,t)} \frac{w_{ij}(t-1)}{ 4 D(i,j,t) } x_j(t-1) \\ 
& = & \left( 1 - \frac{1}{t^{\beta}} \right) x_i(t-1) + \frac{1}{t^{\beta}} \left( 1 - \sum_{j \in S(i,t)} \frac{w_{ij}(t-1)}{ 4 D(i,j,t) } \right) x_i(t-1) + \frac{1}{t^{\beta}}  \sum_{j \in S(i,t)} \frac{w_{ij}(t-1)}{ 4 D(i,j,t)} x_j(t-1)
\end{eqnarray*}  \end{footnotesize} We therefore define $A(t-1)$ in the natural way from the above equation by setting \[ a_{ij}(t-1) = \frac{w_{ij}(t-1)}{ 4 D(i,j,t)} \] for all $i \neq j$ and 
\[ a_{ii}(t-1) = 1 - \sum_{j \in S(i,t)} \frac{w_{ij}(t-1)}{ 4 D(i,j,t) }  \] By construction the rows add up to $1$. The off-diagonal entries are clearly nonnegative, and the matrix is diagonally dominant because Lemma \ref{wbounds} implies that  
\[ 1 - \sum_{j \in S(i,t)} \frac{w_{ij}(t-1)}{ 4 D(i,j,t)} \geq \frac{1}{2} \] This makes $A(t-1)$ nonnegative, stochastic, and diagonally dominant. Finally,  Remark \ref{symw} implies that $A(t-1)$ is symmetric. 
\end{proof}

\bigskip

The next lemma lists some natural Lyapunov functions for our protocol. 
\bigskip

\begin{lemma} $M(t)$ is nonincreasing, $m(t)$ is nondecreasing, and $V_2(x(t))$ is nonincreasing. \label{monotonicity}
\end{lemma} 

\smallskip

\begin{proof} The first two claims follow immediately from the stochasticity of $A(t-1)$. The last one follows from the symmetry
of $A$ - see \cite{noot09} for a proof.  
\end{proof} 

\bigskip We will later require a bound on how large the estimates held by any node can be at any time. The following lemma upper bounds
this quantity in terms of a measure of the initial dispersion. 
\bigskip

\begin{lemma} For all $i,j,t$ such that $\widehat{x}_{i,j, {\rm in}}(t)$ is defined, \[ |\widehat{x}_{i,j, {\rm in}}(t)| \leq ||x(0)||_{\infty}. \] \label{estimatesbound}
\end{lemma}

\smallskip

\begin{proof} By induction, $\widehat{x}_{i,j, {\rm in}}(t)$ belongs to the convex hull 
of $0$ and $x_j(0), x_j(1), x_j(2), \ldots, x_j(t)$. Consequently, its absolute value cannot be larger 
than $\max(|x_j(0)|, |x_j(1)|, \ldots, |x_j(t)|)$. By Lemma \ref{monotonicity}, this is at most $||x(0)||_{\infty}$. 
\end{proof}

\bigskip

We will also need the following observation: there is an upper bound on how much each node moves every step, and that upper bound
decays to zero with time as $1/t^{\beta}$. This is stated formally in the next lemma. 

\bigskip

\begin{lemma} \label{incbound} \[ |x_j(t) - x_j(t-1) | \leq \frac{(1/2) W(0)}{t^{\beta}} \]
\end{lemma} 
  
	\bigskip
	
	\begin{proof} Lemma \ref{monotonicity} implies that for any $t$, \[ \max_{k,l} |x_k(t) - x_l(t)| \leq W(0), \] and then the current lemma follows
	from Eq. (\ref{update}). 	\end{proof}
	
	\bigskip Having recorded a number of facts about our protocol in the previous sequence of lemmas, we now turn to arguing that estimates on 
	core edges get accurate over time. Our starting point is next lemma which makes the following observation: if $q_{j \rightarrow i}(t) \in \{-1,1\}$ and $t$
	is large enough, then
	the quantity $|\widehat{x}_{i,j,{\rm in}}(k) - x_j(k-1)|$, which may be thought of as an ``estimation error'' between the estimate $\widehat{x}_{i,j, {\rm in}}(k)$
	and the true value $x_j(k-1)$, decreases from step $k=t-1$ to step $k=t$ by a multiple of $1/t^{\alpha}$. 
	
	\bigskip
	
	\begin{lemma} If $q_{j \rightarrow i}(t) \in \{-1,1\}$ and $t \geq 1 + (8 W(0))^{\frac{1}{\beta - \alpha}}$ and $t>1$ then 
	\[  |\widehat{x}_{i,j, {\rm in}}(t) - x_j(t-1)|  \leq |\widehat{x}_{i,j, {\rm in}}(t-1) - x_j(t-2)| - \frac{7/8}{t^{\alpha}}  \] \label{decbound}
	\end{lemma} 
	
	\bigskip
	
	\begin{proof} Indeed, under the assumption $q_{j \rightarrow i}(t) \in \{-1,1\}$ we have  \begin{eqnarray*} |\widehat{x}_{i,j, {\rm in}}(t) - x_j(t-1)| & = & |\widehat{x}_{i,j, {\rm in}}(t-1) - x_j(t-1)|  - \frac{1}{t^{\alpha}} \\ 
& \leq & |\widehat{x}_{i,j, {\rm in}}(t-1) - x_j(t-2)| + |x_j(t-1) - x_j(t-2)|  - \frac{1}{t^{\alpha}} \\ 
& \leq &  |\widehat{x}_{i,j, {\rm in}}(t-1) - x_j(t-2)|  + \frac{(1/2) W(0)}{(t-1)^{\beta}} - \frac{1}{t^{\alpha}} 
 \end{eqnarray*}  where the last step used Lemma \ref{incbound}. Now because we are assuming $t$ is large enough so that $t \geq 1+(8W(0))^{\frac{1}{\beta-\alpha}}$ we will then have the upper bound $W(0) \leq (1/8) (t-1)^{\beta-\alpha}$  which allows us to eliminate
 $W(0)$ from the above equation: 
 \[ |\widehat{x}_{i,j, {\rm in}}(t) - x_j(t-1)|  \leq |\widehat{x}_{i,j, {\rm in}}(t-1) - x_j(t-2)| + \frac{1/16}{(t-1)^{\alpha}} - \frac{1}{t^{\alpha}} \leq  |\widehat{x}_{i,j, {\rm in}}(t-1) - x_j(t-2)| - \frac{7/8}{t^{\alpha}} \] where we needed that $t>1$ for the final inequality.  
	\end{proof} 
  
  \bigskip We now proceed to our first substantial lemma, which proves that after a transient period the ``estimation error'' $|\widehat{x}_{i, j, {\rm in}}(t) - x_j(t-1)|$ across core links $(i,j)$ decays to zero at a rate of $O(1/t^{\alpha})$.

\bigskip

\begin{lemma} \label{correctestimate} If $$t \geq  2^\frac{2}{1-\alpha} \lceil ( 32B + 8 B W(0) )^{\frac{1}{\beta-\alpha}} \rceil + \left( 32 B ||x(0)||_{\infty}\right)^{\frac{2}{1-\alpha}} $$ then
\[ | \widehat{x}_{i,j, {\rm in}}(t) - x_j(t-1)| \leq \frac{9/8}{(t-2B-1)^{\alpha}} \] for every pair of nodes $i,j$ such that the edge $(i,j)$ is
a core edge. 
\end{lemma}

\bigskip

\begin{proof} Let $\Lambda$ be the set of times when the core edge $(i,j)$ appears in the graph sequence $G(t)$.   Suppose $z = \lceil (32B + 8 B W(0))^\frac{1}{\beta-\alpha} \rceil-1$ so that in particular $z \geq 1 + (8 B W(0))^\frac{1}{\beta-\alpha}$ and Lemma \ref{decbound} is applicable at all times in $[z,t]$. 

Observe that since $(i,j)$ is a core edge and since $t-z \geq 2B$, the intersection of $\Lambda$ and $[z,t]$ is nonempty. We consider the messages sent across the edge $(i,j)$ to node $i$ and in particular we consider the last time when $q_{j \rightarrow i}(k)=0$ during the interval $k \in [z,t]$. There are three posibilities:  (i) $q_{j \rightarrow i}(k) \in \{-1,1\}$ for all $t \in \Lambda \cap [z,t]$. (ii) $q_{j \rightarrow i}(t') = 0$ for at least 
one $t' \in \Lambda \cap [z,t]$ but the last such $t'$ is less than $t-2B$. (iii) $q_{i \rightarrow j}(t') = 0$ for some $t' \geq t-2B$.

 We will derive an upper bound on the ``estimation error'' $|\widehat{x}_{i,j, {\rm in}}(t) - x_j(t-1)|$ in each of these three cases and put the three inequalities together later. 

\bigskip

\noindent {\bf Case (i):} We have that by Lemma \ref{decbound}, $| \widehat{x}_{i,j,{\rm in}}(k) - x_j(k-1)|$ decreases by at least $(7/8)/k^{\alpha}$ from 
$| \widehat{x}_{i,j, {\rm in}}(k-1) - x_j(k-2)|$ whenever $(i,j)$ appears in $G(k)$; moreover, by Lemma 
\ref{incbound} the same quantity increases by at most $(1/2) W(0)/(k-1)^{\beta}$ whenever $(i,j)$ does not appear in the graph sequence. Since by Lemmas \ref{monotonicity} and \ref{estimatesbound},
$$| \widehat{x}_{i,j,{\rm in}}(z) - x_j(z-1)| \leq 2 ||x(0)||_{\infty}$$ we can consequently conclude that
\begin{equation} \label{firstb} |\widehat{x}_{i,j, {\rm in}}(t) - x_j(t-1)| \leq 2 ||x(0)||_{\infty}  + \sum_{k=z+1}^t \frac{(1/2) W(0)}{(k-1)^{\beta}}  - \sum_{k \in \Lambda \cap [z+1,t]}  \frac{7/8}{k^{\alpha}}. \end{equation} We manipulate this expression by invoking  the assumption that every core edge appears in each block of times $[kB+1, (k+1)B]$. This means that because $\Lambda$ has nonzero intersection with each $[kB+1,(k+1)B]$,  we can infer
 that for any $[a,b]$ such that $a \geq 4B$  and $b-a \geq 2B$ we have\footnote{This inequality follows by some elementary manipulations, which we spell out in this footnote so as to avoid interrupting the flow of the proof. 
 Indeed, note that every interval of length $2B$ or more has at least one element from $\Lambda$. Moreover, since $1/k^{\alpha}$ is a decreasing function of $k$, we have that  $$ \sum_{k \in \Lambda \cap [a,b] } \frac{1}{k^{\alpha}} \geq \sum_{k \geq 1 \mbox{ such that } a + k 2B \leq b} \frac{1}{(a+k2B)^{\alpha}} $$ Now since $a + 2 B \leq (3/2) a$ (due to the assumption $a \geq 4B$), we have that $1/(a+2B)^{\alpha} \geq 1/(1.5 \cdot a)^{\alpha}$ so that 
 $$\sum_{k \in \Lambda \cap [a,b] } \frac{1}{k^{\alpha}} \geq  \left( \frac{1}{2} \frac{1}{(a+2B)^{\alpha}} + \frac{1}{2} \frac{1}{(a+2B)^{\alpha}}  + \sum_{k \geq 2 \mbox{ such that } a + k 2B \leq b}^b \frac{1}{k^{\alpha}} \right) \geq   \sum_{k \geq 0 \mbox{ such that } a + k 2B \leq b} \frac{1}{2 \cdot 1.5^{\alpha} k^{\alpha}} $$ which 
 implies $$ \sum_{k \in \Lambda \cap [a,b]} \frac{1}{k^{\alpha}} \geq \frac{1}{2B} \sum_{k=a}^b \frac{1/3}{k^{\alpha}}.$$ }: 
 \[ \sum_{k \in \Lambda \cap [a,b] } \frac{7/8}{k^{\alpha}}  \geq \frac{1}{8B} \sum_{ k = a}^b \frac{1}{k^{\alpha}}. \] 
Plugging this into Eq. (\ref{firstb}) (which we can do since $z+1 \geq 4B$ and $t \geq z+1+2B$),  
\[ |\widehat{x}_{i,j, {\rm in}}(t) - x_j(t-1)| \leq 2 ||x(0)||_{\infty}  + \sum_{k=z+1}^t \frac{(1/2) W(0)}{(k-1)^{\beta}}  - \sum_{k=z+1}^t  \frac{1/(8B)}{k^{\alpha}} \]
Now since the inequality $z \geq 1 + (8B W(0))^{\frac{1}{\beta - \alpha}}$ implies that $W(0) \leq  z^{\beta - \alpha}/(8B) \leq (k-1)^{\beta - \alpha}/(8B)$ for all $k \in [z+1,t]$, we have: 

\begin{eqnarray*} |\widehat{x}_{i,j, {\rm in}}(t) - x_j(t-1)| &  \leq  &  2 ||x(0)||_{\infty} + \sum_{k=z+1}^t \frac{1}{16 B (k-1)^{\alpha} } - \sum_{k=z+1}^t \frac{1}{8 B k^{\alpha}} \\
& \leq & 2 ||x(0)||_{\infty} + \frac{1}{16 B z^{\alpha}}  - \sum_{k=z+1}^t  \frac{1/(16B)}{k^{\alpha}} \\
& \leq & 2 ||x(0)||_{\infty} + \frac{1}{16 B}  - \sum_{k=z+1}^t  \frac{1/(16B)}{k^{\alpha}}. 
\end{eqnarray*}

\smallskip 

\noindent {\bf Case (ii):} The only difference from case $(i)$ is that rather than bounding the change in $ |\widehat{x}_{i,j, {\rm in}}(k) - x_j(k-1)|$ 
from time $z$ to time $t$, we bound it on from time $t'$ to time $t$. Indeed, since $q_{j \rightarrow i}(t')=0$, we have that 
\[ |\widehat{x}_{i,j, {\rm in}}(t') - x_j(t'-1)| = |\widehat{x}_{i,j, {\rm in}}(t'-1) - x_j(t'-1)| \leq \frac{1}{(t')^{\alpha}} \] and 
consequently 
\[ | \widehat{x}_{i,j, {\rm in}}(t) - x_j(t-1)| \leq \frac{1}{(t')^{\alpha}} + \sum_{k = t'+1}^t \frac{(1/2) W(0)}{(k-1)^{\beta}} - \sum_{k \in \Lambda \cap [t'+1,t]} \frac{7/8}{k^{\alpha}} \] Since $t' \geq z \geq 4B$ and $t-t' \geq 2B$, we can apply an identical argument as in case (i) to obtain
\[ |\widehat{x}_{i,j, {\rm in}}(t) - x_j(t-1)| \leq \frac{1+1/(16B)}{(t')^{\alpha}}  - \sum_{k=t'+1}^t  \frac{1/(16B)}{k^{\alpha}}. \]

\smallskip 

\noindent {\bf Case (iii):} We proceed as in case (ii) by bounding the change in $ |\widehat{x}_{i,j, {\rm in}}(k) - x_j(k-1)|$  from time $t'$ to time $t$. In this case, however, we rely on the fact that $t - t' \leq 2B$, so that a straightforward application of Lemma \ref{incbound} as well as the bound $W(0) \leq (1/(8B)) (t-2B)^{\beta - \alpha}$ gives: 
\[ |\widehat{x}_{i,j, {\rm in}}(t) - x_j(t-1)| \leq \frac{1}{(t-2B)^{\alpha}} + 2B \frac{(1/2)W(0)}{(t-2B)^{\beta}} \leq \frac{9/8}{(t-2B)^{\alpha}}. \] 

\smallskip

Putting it all together, we have:  \begin{footnotesize} \begin{equation} \label{withsumbound}
 |\widehat{x}_{i,j, {\rm in}}(t) - x_j(t-1)| \leq  \max \left( 2 ||x(0)||_{\infty} + \frac{1}{16B} - \sum_{k=  z+1}^t \frac{1/16}{B k^{\alpha}},~~ \max_{t'=  z, \ldots,t-2B-1} \frac{17/16}{{t'}^{\alpha}} - \sum_{k=t'+1}^t \frac{1/16}{B k^{\alpha}}, ~~\frac{9/8}{(t-2B)^{\alpha}} \right) \end{equation} \end{footnotesize} Our next step is to lower bound the sums appearing on the right-hand side of the above equation. Using the standard method of lower-bounding the sum of a nonincreasing function by an integral, we have
 \begin{equation} \label{sqrtbound}  \sum_{k=a}^{b} \frac{1}{k^{\alpha}} \geq \frac{ {(b+1)^{1-\alpha}-a^{1-\alpha}} }{1-\alpha} = \frac{((b+1)^{\frac{1-\alpha}{2}}+a^{\frac{1-\alpha}{2}})((b+1)^{\frac{1-\alpha}{2}}-a^{\frac{1-\alpha}{2}})}{1-\alpha}. \end{equation} Now the lower bound on $t$ we have assumed in the statement of this lemma implies that $$t^{(1-\alpha)/2} - (z+1)^{(1-\alpha)/2} \geq 2 \left( \lceil (32B + 8 B W(0))^\frac{1}{\beta-\alpha} \rceil \right)^{\frac{1-\alpha}{2}}  - \left( \lceil (32B + 8 B W(0))^\frac{1}{\beta-\alpha} \rceil \right)^{\frac{1-\alpha}{2}} \geq 1.$$ Moreover, $t^{(1-\alpha)/2} + (z+1)^{(1-\alpha)/2} \geq 32 B ||x(0)||_{\infty} + 1$. Plugging these two inequalities into Eq. (\ref{sqrtbound}) we get  that the first term in the maximization of Eq. (\ref{withsumbound}) is nonpositive. 

As for the second term 
 in the maximization, it is therefore bounded as 
 \[     \max_{t'= z, \ldots,t-2B-1} \frac{17/16}{(t')^{\alpha}} - \frac{1}{16B} \frac{((t+1)^{1-\alpha} - (t'+1)^{1-\alpha})}{1-\alpha}  \] and since $t' \geq 32B$ for all the values of $t'$ over which we are maximizing, the expression inside the maximum is 
 an increasing function\footnote{Indeed, if $f(t) = c/t^{\alpha} + d (t+1)^{1-\alpha}$ where $\alpha \in (0,1)$  then $$f'(t) = - \alpha c t^{-\alpha-1} + d (1-\alpha) (t+1)^{-\alpha}.$$ The latter expression is nonnegative whenever $\frac{d (1-\alpha)}{c \alpha} \frac{t^{\alpha+1}}{(t+1)^{\alpha}} \geq 1$. Now in the
above expression, $c=17/16, d= 1/(16B(1-\alpha))$ so we must verify that $$ \frac{1}{17B \alpha} t \left( \frac{t}{t+1} \right)^{\alpha} \geq 1. $$  This holds
if $t \geq 32B$ because $(32/17)(32/33) \geq 1.$} of $t'$.   Consequently, we may choose $t'=t-2B-1$ in the second term, which immediately implies the current lemma. 
\end{proof}

\bigskip

We next demonstrate a consequence of the previous lemma: if $t$ is large enough and the message $q_{i \rightarrow j}$ on a core
edge is nonzero at 
some time, then the same message $q_{i \rightarrow j}$ will be zero for a substantial amount of time afterwards whenever the edge $(i,j)$ appears. Intuitively,
this lemma says that messages across a core edge cannot be in $\{-1,1\}$ almost always: at least a constant fraction of the time, a message across
a core edge should equal zero. 
\bigskip

\begin{lemma} \label{messagesexchangepre} Suppose  $$t \geq  2^{\frac{2}{1-\alpha}} \lceil 32B + 8 W(0)  \rceil^{\frac{1}{\beta-\alpha}} + \left( 32 B ||x(0)||_{\infty} \right)^{\frac{2}{1-\alpha}} + 1$$ and the edge $(i,j)$ is a core edge. If 
$q_{i \rightarrow j}(t) \neq 0$ then $q_{i \rightarrow j}(k)=0$ for all $k \in [t+1,t+3B]$ during which $(i,j)$ 
appears. \end{lemma}

\bigskip

\begin{proof} By Lemma \ref{correctestimate}, 
\begin{equation} \label{estappl} | \widehat{x}_{i,j,{\rm in}}(t-1) - x_j(t-2)| \leq \frac{9/8}{(t-2B-2)^{\alpha}}.\end{equation} Consequently, $q_{i \rightarrow j}(t) \neq 0$ implies that 
for each $k \in [t,
t+3B-1]$

\begin{eqnarray} \label{placeholder1} |\widehat{x}_{i,j, {\rm in}} (k) - x_j(k)| & \leq & | \widehat{x}_{i,j,{\rm in}}(k) - x_j(k-1)| + \frac{(1/2)W(0)}{k^{\beta}}  \nonumber \\
& \leq &  |\widehat{x}_{i,j,{\rm in}}(t) - x_j(t-1)| + 3B\frac{(1/2)W(0)}{t^{\beta}} \nonumber \\
& \leq & |\widehat{x}_{i,j,{\rm in}}(t-1) - x_j(t-2)| - \frac{7/8}{t^{\alpha}} + 3B \frac{(1/2)W(0)}{t^{\beta}}  \nonumber \\
& \leq & \frac{9/8}{(t-2B-2)^{\alpha}} - \frac{7/8}{t^{\alpha}} + 3B \frac{(1/2) W(0)}{t^{\beta}} \label{expr}.
\end{eqnarray} Here the first two inequalities follow from Lemma \ref{incbound}; the third follows from Eq. (\ref{decbound}); and the last one follows by plugging in Eq. (\ref{estappl}). 

 We claim that our assumption $t \geq   2^{\frac{1}{1-\alpha}} \lceil 32B + 8B W(0)\rceil^{\frac{1}{\beta-\alpha}} + \left( 32 B ||x(0)||_{\infty} \right)^{\frac{2}{1-\alpha}} + 1$ implies the final expression on the right is always 
less than $1/(t+3B)^{\alpha}$. Indeed, this assumption implies that $t \geq 32B$ so that $t-2B-2 \geq 0.85t$ and $t+3B \leq 1.1t$. Consequently, 
\begin{eqnarray*} \frac{1}{(t+3B)^{\alpha}} - \frac{9/8}{(t-2B-2)^{\alpha}}+\frac{7/8}{t^{\alpha}} & \geq & \frac{1/1.1^{\alpha}-(9/8)/0.85^{\alpha}+7/8}{t^{\alpha}} \\
& > & \frac{3/8}{ t^{\alpha}} \\
& = & \frac{3/8}{t^{\beta}} t^{\beta-\alpha} \\ 
& \geq & \frac{3/2}{t^{\beta}} B W(0)
\end{eqnarray*} where the last step used our assumed lower bound on $t$ which implies that $t^{\beta-\alpha} \geq 4 B W(0)$. The last inequality
is, after rearranging, precisely our claim that the right-hand side of Eq. (\ref{expr}) is at most $1/(t+3B)^{\alpha}$.

Since  $$|\widehat{x}_{i,j, {\rm in}} (k) - x_j(k)| \leq \frac{1}{(t+3B)^{\alpha}} \leq \frac{1}{k^{\alpha}}$$ for all $k \in [t,t+3B-1]$, we have that $q_{i \rightarrow j}(k)=0$ for all $k \in [t+1, t+3B]$  when the edge $(i,j)$ appears. 
\end{proof} 

\bigskip

A corollary of this lemma is that every $3B$ steps, both messages $q_{i \rightarrow j}$ and $q_{j \rightarrow i}$ across a 
core link will be zero.  This is stated formally in the following lemma. This is a key lemma for us, and we will use it later to argue
that as long as the values $x_i(t)$ are not very close together, a $V_2(t)$-reducing Metropolis update should eventually take place. 


\bigskip

\begin{corollary} \label{messagesexchange} Suppose the edge $(i,j)$ is a core edge and 
suppose $t$ is a multiple of $B$ satisfying $$t \geq  2^{\frac{2}{1-\alpha}} \lceil 32 B + 8 B W(0)  \rceil^{\frac{1}{\beta-\alpha}} + \left( 32 B ||x(0)||_{\infty} \right)^{\frac{2}{1-\alpha}}.$$ Then there exists a time in $[t+1,t+3B]$ with 
\[ q_{i \rightarrow j}(t) = q_{j \rightarrow i}(t) = 0. \]  \end{corollary}
\begin{proof} Indeed, since the edge $(i,j)$ appears at least once in every period $[kB+1, (k+1)B]$, it follows it must appear at least thrice within the period 
$[t+1,t+3B]$. If both $q_{i \rightarrow j}$ and $q_{j \rightarrow i}$ are zero during the first of those times, we are finished. If not, Lemma \ref{messagesexchangepre} implies they will be zero during either the second or the third of these times. 
\end{proof}

\bigskip

Our next lemma forms the core of the proof our main theorem. Informally, it states that if the variance is not too low, it must
decrease by a multiplicative factor over the next $3B$ steps. Before proceeding, we need to mention the standard relationship between $W(t)$ and $V_2(t)$ which
we will use:  \begin{equation} \label{infty2ineq} \frac{W(t)}{2} \leq V_2(t) \leq \sqrt{n} W(t). 
  \end{equation}
\bigskip

\begin{lemma} Suppose that $t$ is a multiple of $B$ satisfying \[ t \geq  2^{\frac{2}{1-\alpha}} \lceil 32B + 8 B W(0)  \rceil^{\frac{1}{\beta-\alpha}} + \left( 32 B ||x(0)||_{\infty} \right)^{\frac{2}{1-\alpha}}  \] and \begin{equation} V_2(t) \geq \frac{8n^{1.5}}{t^{\alpha}}. \label{assumption} \end{equation} Then 
\[ V_2(t+3B) \leq \left( 1-\frac{1}{50 n^3 D (t+3B)^{\beta}} \right) V_2(t) \] \label{mainlemma}
\end{lemma}

\begin{proof} Because the subgraph consisting of core edges is connected, we have that at any time $t$ there exists a core edge $(i,j)$ such that $| x_i(t) - x_j(t) | \geq W(t)/n$. Indeed, if it were otherwise, then taking the shortest path along
core edges from a node with the smallest value to a 
node with the largest value and considering the change in the value along each edge we would get that 
$M(t) - m(t) < n \left( W(t)/n \right) = W(t)$, which is nonsense. 

Thus we can conclude that for some core edge $(i,j)$,
\begin{equation} \label{nonefive} |x_i(t) - x_j(t)| \geq \frac{W(t)}{n} \geq \frac{V_2(t)}{n^{1.5}}. \end{equation}  Now applying Eq. (\ref{infty2ineq}) to  Eq. (\ref{assumption}) we have, \[ W(t) > 8 \frac{n}{t^{\alpha}}. \] 

Putting the last pair of  inequalities together,  \begin{equation} \label{bigdiff} |x_i(t)-x_j(t)| \geq \frac{W(t)}{n} > \frac{8n/t^{\alpha}}{n} = \frac{8}{t^{\alpha}}. \end{equation} 

Now our assumed lower bound $t \geq  2^{\frac{2}{1-\alpha}} \lceil 32B + 8 B W(0) \rceil^{\frac{1}{\beta-\alpha}} + \left( 32 B ||x(0)||_{\infty} \right)^{\frac{2}{1-\alpha}}$
allows us to apply Corollary \ref{messagesexchange} and claim there exists a time $t'$ in $[t+1,t+3B] $ such that $q_{i \rightarrow j}(t')=q_{j \rightarrow i}(t')=0$. Moreover, since between times $t$ and $t'-1$, both $x_i(t)$ and $x_j(t)$ can change by at most $3B(1/2)W(0)/t^{\beta}$, and our assumed lower bound on 
$t$ implies that this is at most $(1/2)/t^{\alpha}$. This means $|x_i(t'-1)-x_j(t'-1)| > 7/t^{\alpha} > 7/(t')^{\alpha}$ and so $$|\widehat{x}_{i,j, {\rm in}}(t') - \widehat{x}_{i,j, {\rm out}}(t')| > \frac{7}{(t')^{\alpha}} - 2 \frac{9/8}{(t'-2B-1)^{\alpha}} > \frac{4}{(t')^{\alpha}},$$ where we used $t \geq 32B$ for the final inequality. The above 
equation implies that $j \in S(i,t')$ and consequently $i \in S(j,t')$. 

Moreover, putting Eq. (\ref{bigdiff}) with the observation in the previous paragraph that both $|x_i(t) - x_i(t'-1)|$ and $|x_j(t) - x_j(t'-1)|$ are upper bounded by
$(1/2)/t^{\alpha}$, we have that 
\begin{equation} \label{primevar} |x_i(t'-1)-x_j(t'-1)| \geq \frac{7}{8} |x_i(t) - x_j(t)| \geq \frac{7}{8} \frac{V_2(t)}{n^{1.5}} \geq \frac{7}{8} \frac{V_2(t'-1)}{n^{1.5}} \end{equation} where the last inequality used the monotonicity of $V_2(t)$ and the penultimate inequality used Eq. (\ref{nonefive}). 

We now proceed to lower bound the decrease from $V_2(t'-1)$ to $V_2(t')$ due to the link 
between $i$ and $j$ at time $t'$. 

Since $j \in S(i,t')$ we have that $a_{ij}(t'-1)$ is positive, and therefore $a_{ij}(t'-1) \geq 1/(8 D(i,j,t'))$. Using 
the decomposition (see \cite{noot09}, \cite{xbk07})
\[ A^2(t'-1) = I - \sum_{k<l} [A^2(t'-1)]_{kl} (\e_k - \e_l) (\e_k - \e_l)^T,\]  where $\e_k$ as usual means the $k$'th basis vector, we have
that \begin{eqnarray*} V_2(A(t'-1)x(t'-1))^2 & = & (x(t'-1) - \bar{x}(t'-1) \1)^T (A(t'-1))^2 (x(t'-1) - \bar{x} \1) \\ 
& = &  V_2(x(t'-1))^2 - \sum_{k<l} [A^2(t'-1)]_{kl} (t) (x_k(t'-1) - x_l(t'-1))^2, 
\end{eqnarray*} and, since $A(t'-1)$ is diagonally dominant and satisfies $a_{ij}(t'-1) \geq 1/(8 D(i,j,t'))$, we have that
$[A^2(t'-1)]_{ij} \geq (1/(16  D(i,j,t') )$; moreover, using Eq. (\ref{primevar}), we obtain
\begin{eqnarray*} V_2(A(t'-1)x(t'-1))^2 & \leq & V_2(x(t'-1))^2  -  \frac{(7/8)^2}{16 D(i,j,t')} \frac{V_2(t'-1)^2}{n^3} 
\end{eqnarray*} 
so
\begin{equation} \label{multbound} V_2(A(t'-1) x(t'-1)) \leq V_2(t'-1) \sqrt{1 - \frac{(7/8)^2}{16 n^3 D}} \leq V_2(x(t'-1)) (1 - \frac{1}{50n^3 D}) 
\end{equation} and therefore
\begin{eqnarray*} V_2(t') & \leq &   \left( (1-\frac{1}{(t')^{\beta}}) V_2 (t'-1) + \frac{1}{(t')^{\beta}} V_2 (A(t'-1) x(t'-1) ) \right)  \\ 
& \leq & \left( 1-\frac{1}{50 n^3 D (t')^{\beta}} \right) V_2(x(t'-1))
\end{eqnarray*} and since $V_2(t)$ is nonincreasing in $t$ (Lemma \ref{monotonicity}) and $t' \leq t+3B$, this concludes the proof. 
\end{proof}

\bigskip

With this quantitative bound on the decrease of $V_2(t)$, we are almost ready to begin the proof of our main result,  Theorem \ref{precursorthm}. However, we first need the following technical lemma which will allow us to bound some of the expressions which will appear in the course of the proof that theorem.

\bigskip

\begin{lemma} The function $g(k) = \frac{e^{(1/d) (k+f)^c}}{k^e}$ with $c \in (0,1)$, where the constants $d,e$ and $f$ are positive,  is nondecreasing over the range $ k \in [t_0, +\infty)$ where $t_0$
is any positive number which satisfies $\frac{t_0}{(t_0+f)^{1-c}} \geq de/c$. \label{nondeclemma}
\end{lemma}

\bigskip

\begin{proof} We have \[ g'(k) = \frac{e^{(1/d) (k+f)^c} (c/d) (k+f)^{c-1} k^{e} - e k^{e-1} e^{(1/d)(k+f)^c} }{k^{2e}} \] so
$g'(k) \geq 0$ if $$\frac{c}{d} (k+f)^{c-1} \geq \frac{e}{k}$$ or $$\frac{k}{(k+f)^{1-c}} \geq \frac{de}{c}.$$ Moreover, the 
expression on the left-hand side is nondecreasing for nonnegative $k$. 
\end{proof} 

\bigskip

We now proceed to the proof of our main result. Following the numerous lemmas which we have already established,  the proof now follows by a
straightforward (albeit tedious) calculation. Indeed, Lemma  \ref{mainlemma} tells us that either $V_2(t)$ is small in the sense of being bounded above
by something that goes to zero with time, or $V_2(t)$ decreases by a multiplicative 
factor. It is now an exercise in analysis to use this dichotomy to obtain an upper bound on the rate at which $V_2(t)$ approaches zero. 

\bigskip

\begin{proof}[Proof of Theorem \ref{precursorthm}] Let us use the shorthand $$T=2^{\frac{2}{1-\alpha}} \lceil 32B + 8B W(0) \rceil^{\frac{1}{\beta-\alpha}} + \left( 32 B ||x(0)||_{\infty} \right)^{\frac{2}{1-\alpha}} + (300 n^3 D B )^{\frac{1}{1-\beta}}$$ and let us suppose $l$ is a multiple of $B$ which satisfies $l \geq T$.  If Eq. (\ref{assumption}) holds at times 
$l,l+3B, l+6B, \ldots, l+3(k-1)B$ for some $k \geq 1$, then we can apply Lemma \ref{mainlemma} to get 
\[ V_2(l+3kB) \leq V_2(l)  \left( 1-\frac{1}{50 n^3 D (l+3B)^{\beta}} \right) \left( 1-\frac{1}{50n^3 D (l+6B)^{\beta}} \right) \cdots \left( 1-\frac{1}{50n^3 D (l+3kB)^{\beta}} \right) \] which, using the inequality $\ln(1-x) \leq -x$, implies 
\begin{eqnarray*} \ln \frac{ V_2(l+3kB)}{V_2(l)} & \leq & - \sum_{j=1}^{k} \frac{1}{50n^3 D (l + 3jB)^{\beta}}  \\ 
& \leq & - \frac{1}{50n^3 D} \frac{1}{3B} \sum_{m=l+3B}^{l+3kB} \frac{1}{m^{\beta}} \\
& \leq & - \frac{1}{150n^3 D } \int_{l+3B}^{l+3kB} \frac{dz}{z^\beta} \\  
& \leq & - \frac{(l+3kB)^{1-\beta}-(l+3B)^{1-\beta}}{150 B  n^3 D (1 - \beta)} 
\end{eqnarray*} which in turn implies
\begin{equation} \label{alwaysmult} V_2(l+3kB) \leq V_2(l) e^{-((l+3kB)^{1-\beta}-(l+3B)^{1-\beta})/(150(1-\beta)n^3 D B)} \end{equation} Now fix $T'$ to be the first multiple of $B$ which is at least $T$, and observe that $t$ (which satisfies the lower bound we have assumed in the theorem statement) is at least $T'$. We now consider the last time in the set $\Gamma = \{T', T'+3B, T'+6B, T'+9B, \ldots \} \cap \{1, \ldots, t\}$ when Eq. (\ref{assumption}) holds. We consider three possibilities. 

It may be that Eq. (\ref{assumption}) holds for all numbers in $\Gamma$. In that case, Eq. (\ref{alwaysmult}) implies
\[ V_2(t) \leq \ V_{2}(0) e^{-((t-3B)^{1-\beta}-(T'+3B)^{1-\beta})/(150(1-\beta) n^3 D B)}. \] On the other hand, it may be that the last number in $\Gamma$ when Eq. (\ref{assumption}) does not hold occurs at time $t-7B$ or afterwards. In that case, 
\[ V_2(t) \leq \frac{8 n^{1.5}}{(t-7B)^{\alpha}}. \] Finally, it may be that there exists a number $k \in \Gamma$ where Eq. (\ref{assumption}) does not hold, but the last such $k$ occurs before $t-7B$. In that case, by considering the variance decay from the next multiple of $B$ after $k$ and applying Lemma \ref{mainlemma}, we obtain
\[ V_2(t) \leq \max_{T' \leq k \leq t-7B} \frac{8n^{1.5}}{k^{\alpha}} e^{-((t-3B)^{1-\beta}-(k+4B)^{1-\beta})/(150(1-\beta) n^3 D B)}. \] Putting the last
three inequalities together, we have the
unconditional bound  \begin{small}
\begin{equation} \label{unconditional} V_2(x(t)) \leq \max \left( V_{2}(0) e^{-((t-3B)^{1-\beta}-(T'+3B)^{1-\beta})/(150(1-\beta) n^3 D B)}, \max_{T' \leq k \leq t-7B} \frac{8n^{1.5}}{k^{\alpha}} e^{-((t-3B)^{1-\beta}-(k+4B)^{1-\beta})/(150(1-\beta) n^3 D B)}, \frac{8n^{1.5}}{(t-7B)^{\alpha}} \right)  \end{equation} \end{small}  Lets consider how long it takes the first term to fall below $\epsilon$.  The 
inequality \[ V_{2}(0) e^{-((t-3B)^{1-\beta}-(T'+3B)^{1-\beta})/(150(1-\beta) n^3 D B)} \leq \epsilon \] is implied by
\[ (t-3B)^{1-\beta}-(T'+3B)^{1-\beta} \geq 150 n^3 D  B \log \frac{V_2(0)}{\epsilon} \] or 
\[ t \geq 3B + \left( (T'+3B)^{1-\beta} + 150n^3 D B \log \frac{V_2(0)}{\epsilon} \right)^{\frac{1}{1-\beta}} \] which, using the inequality $(a+b)^k \leq 2^{k - 1} (a^k+b^k)$ for $k \geq 1$, is in turn implied by
\begin{eqnarray*} t & \geq &  3B + 2^{\frac{1}{1-\beta}} \left( T' + 3B + (150 n^3 D B \log \frac{\epsilon}{V(0)})^{\frac{1}{1-\beta}} \right) \end{eqnarray*}
which is implied by \[ t
\geq  3B + 2^{\frac{1}{1-\beta}} \left( T + 4B + (150 n^3 D B \log \frac{\epsilon}{V(0)})^{\frac{1}{1-\beta}} \right)
\]
Next we consider the second term in the maximum of Eq. (\ref{unconditional}). We will argue that it is maximized at the choice of $k=t-7B$. Indeed, since $t$ is fixed and the maximization is over $k$, we are maximizing the function 
\[ \frac{ e^{(k+4B)^{1-\beta}/(150 (1-\beta)n^3 B )} }{k^{\alpha}}  \] We now appeal to Lemma \ref{nondeclemma}. In the notation of that lemma, $c=1-\beta, d = 150 (1-\beta) n^3 B, e = \alpha, f=4B$. The lemma then tells us that if $T'$ is large enough so that
\begin{equation} \label{tcond} \frac{T'}{(T'+4B)^{\beta}} \geq 150 n^3 D B \alpha \end{equation} we may set $k=t-7B$ in the maximization problem. But since $T' \geq 32B$, we have
that \[ \frac{T'}{(T'+4B)^{\beta}} = (T')^{1-\beta} \frac{(T')^{\beta}}{(T'+4B)^{\beta}} \geq \frac{(T')^{1-\beta}}{1.2^{\beta}} \] so that for Eq. (\ref{tcond}) to be true it suffices that $T' \geq (1.2^{\beta} \cdot 150 n^3 D B \alpha)^{1/(1-\beta)}$ which clearly holds by definition of $T$. 
As a consequence, we finally have $V_2(t)$ is below $\epsilon$ if \begin{footnotesize} 
\[ t \geq 7B + \max \left( 2^{\frac{1}{1-\beta}} \left( T + 4B + (150 n^3 D B \log \frac{\epsilon}{V(0)})^{\frac{1}{1-\beta}} \right), \left(  \frac{8 n^{1.5}}{\epsilon} \right)^{\frac{1}{\alpha}} \right) \] \end{footnotesize}  After some simple algebra and rearrangement of terms, this becomes the bound of 
Theorem \ref{precursorthm}.

\end{proof}

\section{Simulations\label{simulations}} Here we report on some simulations of our consensus protocol. We will be focusing on seeing 
how the performance of our protocol scales with the time $t$ for fixed $n$ and seeing how the time to reach a certain level of 
accuracy scales with the number of nodes for various graph topologies. 

We will be simulating a slightly modified form of our protocol in which we omit the stepsize of $1/t^{\beta}$, i.e, in which
we set $\beta=0$. 
The introduction of this stepsize is a mathematical necessity for our convergence result, but it does not appear to be practically necessary. Similarly, we will double the weight every agent places on its neighbors as our consensus protocol is likely too conservative.  Thus we will be simulating the update 
\[ x_i(t) =    x_i(t-1) +  \sum_{j \in S(i,t)} \frac{\widehat{x}_{i,j,{\rm in}}(t) - \widehat{x}_{i,j, {\rm out}}(t)} {2 {\rm max} (d_i(t), d_j(t))} \] instead of Eq. (\ref{mainiter2}). 
\begin{figure}[h] 
\begin{center}$
\begin{array}{ccc}
\includegraphics[width=3.2in]{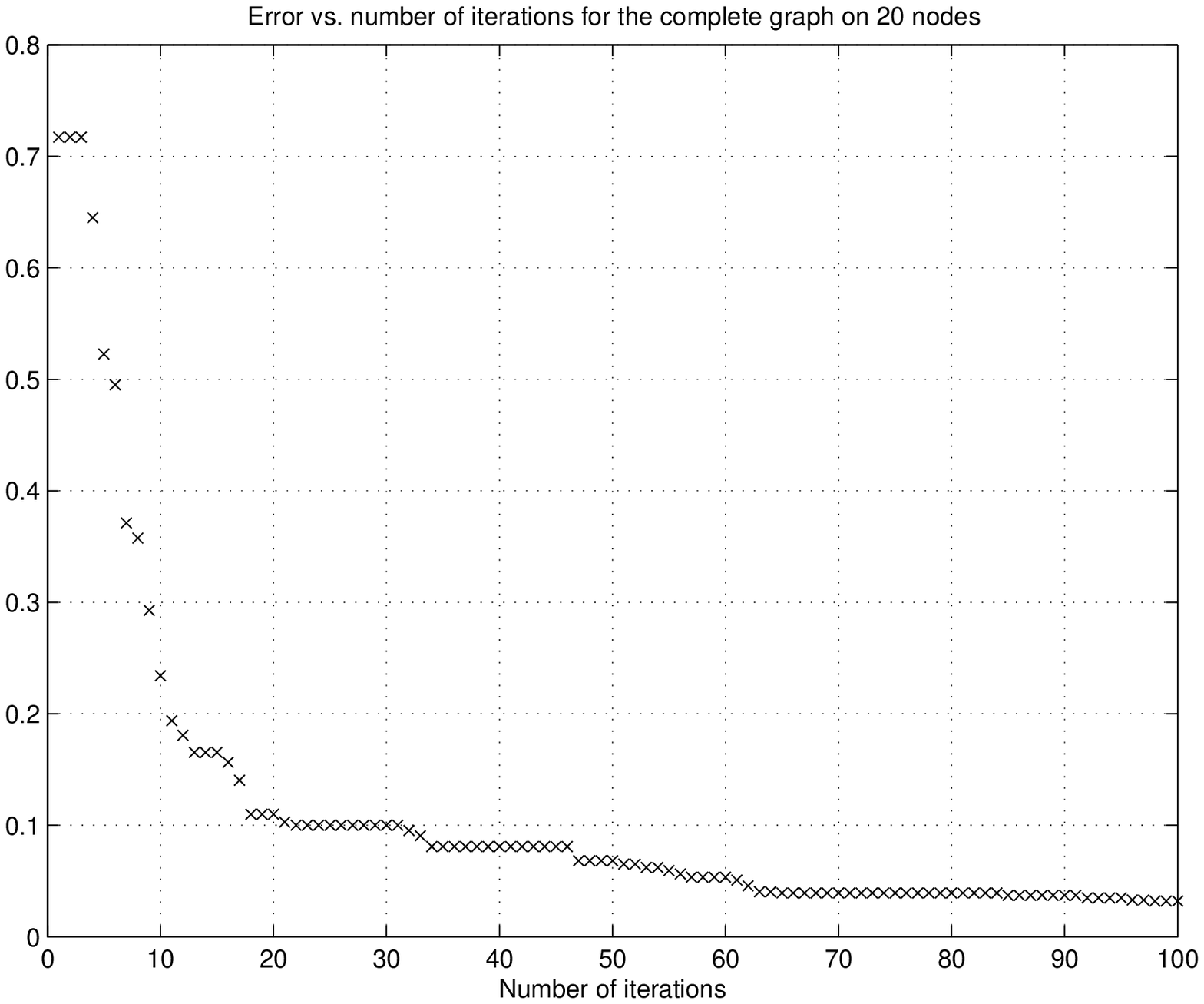} &
\includegraphics[width=3.2in]{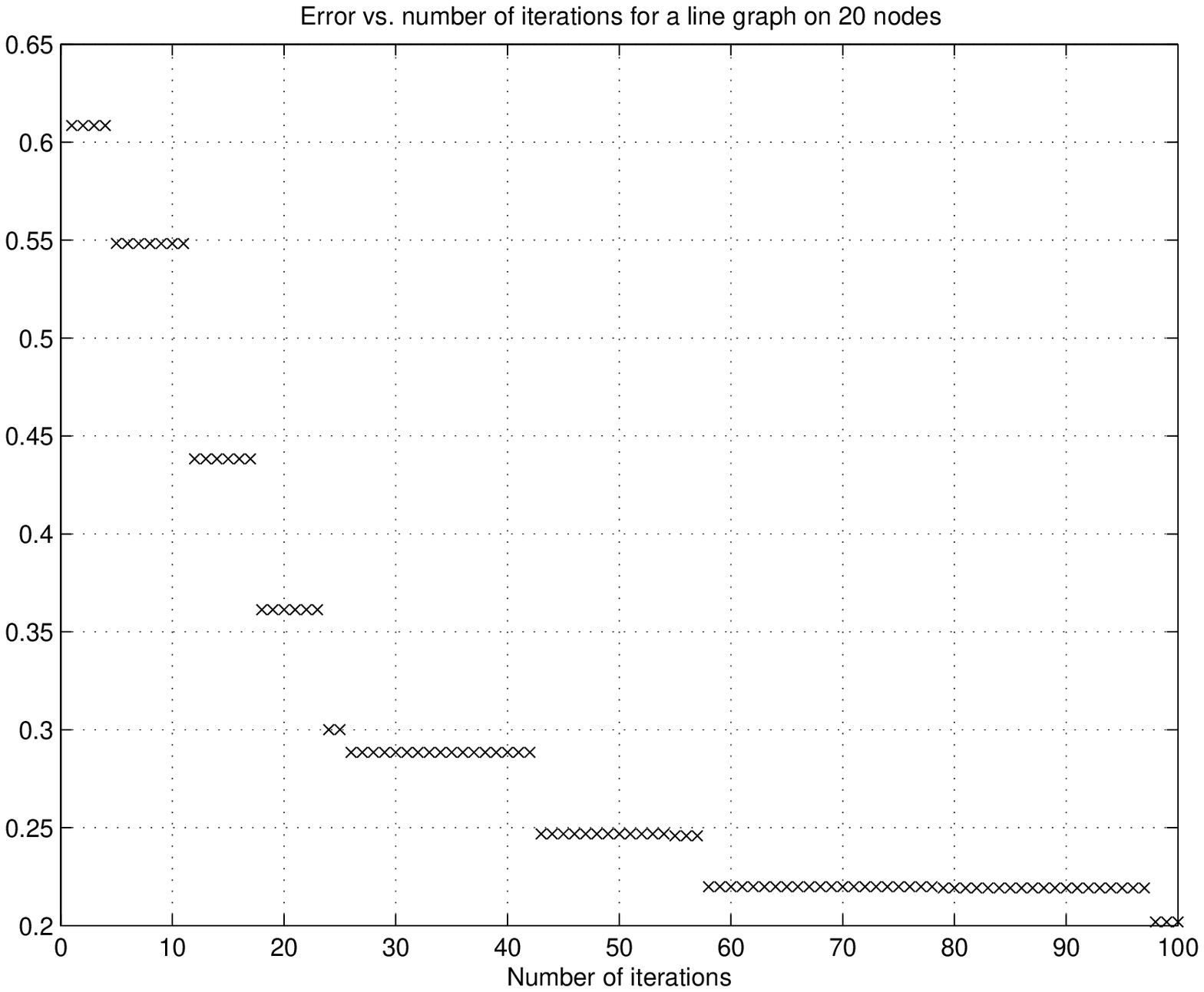} 
\end{array}$ 
\end{center} \caption{The left plot shows our consensus protocol on the complete graph on $20$ nodes, while the right graph shows our consensus protocol on the line graph on $20$ nodes. Both graphs show the error at time $t$, i.e., $\max_i |x_i(t) - \frac{1}{n} \sum_{j=1}^n x_j(0)|$ vs
the number of iterations $t$. Initial conditions are random for both graphs, and we have chosen $\alpha=0.9$ in both.}  
\label{time-error}
\end{figure}

In Figure \ref{time-error}, we show the decay of the error $\max_i |x_i(t) - \frac{1}{n} \sum_{j=1}^n x_j(0)|$ with time for the 
complete graph and the line graph with $\alpha=0.9$. Much like our error bound predicts, we see a slow decay which appears to be asymptotic
to a power of $t$. We point out the jagged features of the graph, which are more prominent in the case of the line: it often takes some time
for the estimates $\widehat{x}_{i,j, {\rm in}}$ to become accurate,  which results in periods without any updates. 

\begin{figure}[h] 
\begin{center}$
\begin{array}{ccc}
\includegraphics[width=3.2in]{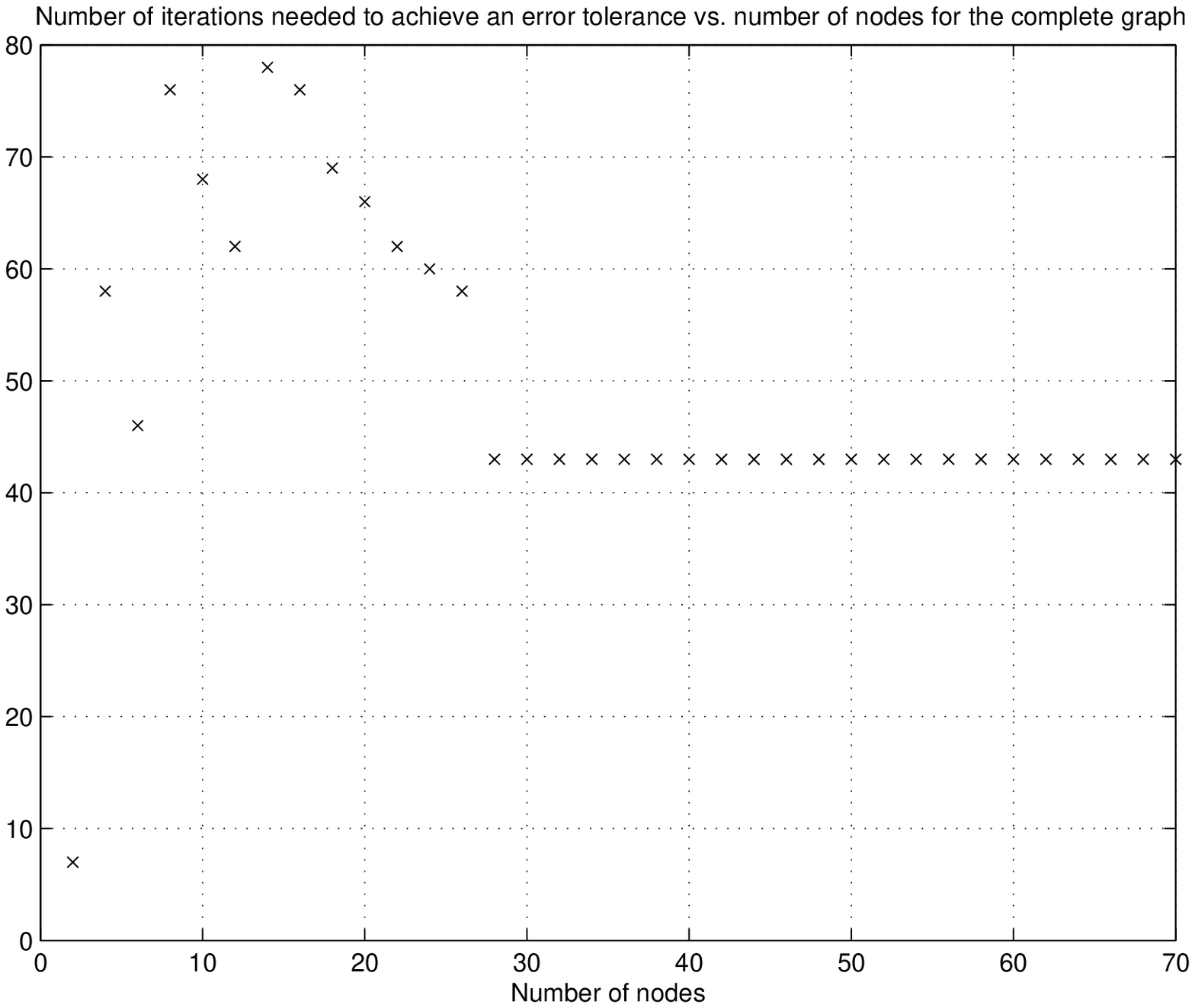} &
\includegraphics[width=3.2in]{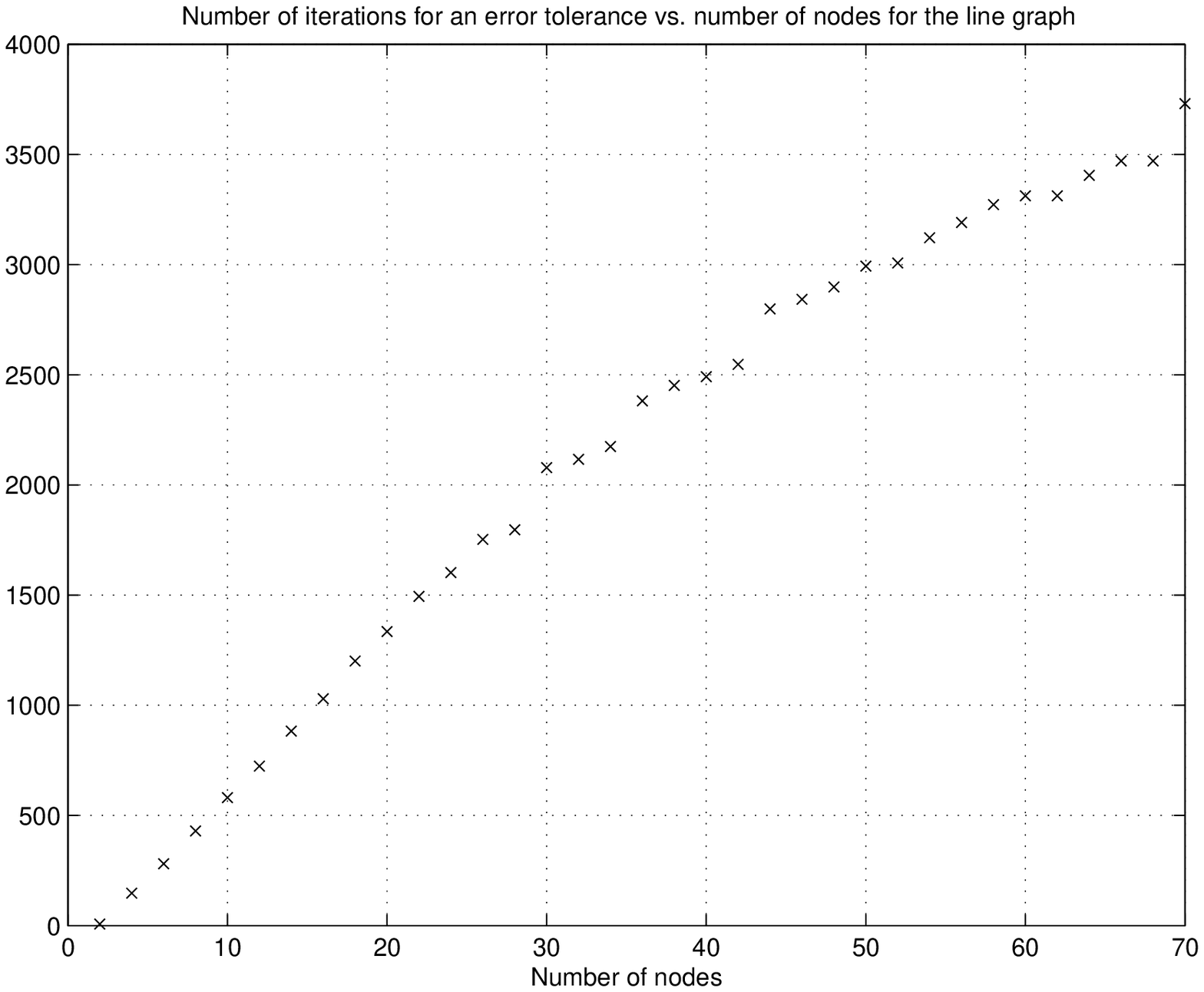} 
\end{array}$ 
\end{center} \caption{Both graphs show the time it takes for our protocol to reach an error of $0.05$, i.e., the time until 
$\max_i |x_i(t) - \frac{1}{n} \sum_{j=1}^n x_j(0)| \leq 0.05$. The number of nodes is shown on the x-axis. The left picture is
for the complete graph while the right picture is for the line graph. The initial condition is 
the vector with $x_1(0)=1$ and $x_k(0)=1$ for $k=2, \ldots, n$  in both cases, and as before, $\alpha=0.9$.}  
\label{nscaling}
\end{figure}

In the next Figure \ref{nscaling}, we are concerned with the time it takes our consensus protocol to achieve a certain level of error;
we plot the time until $\max_i |x_i(t) - \frac{1}{n} \sum_{j=1}^n x_j(0)| \leq 0.05$ for the complete 
graph and the line graph. As expected, for the complete graph, the time does not grow with $n$. We note how irregular this time is 
for the complete graph.  Finally, we see a reasonably slow growth with $n$ for the case of the line.

\begin{figure}[h] 
\begin{center}$
\begin{array}{ccc}
\includegraphics[width=3.2in]{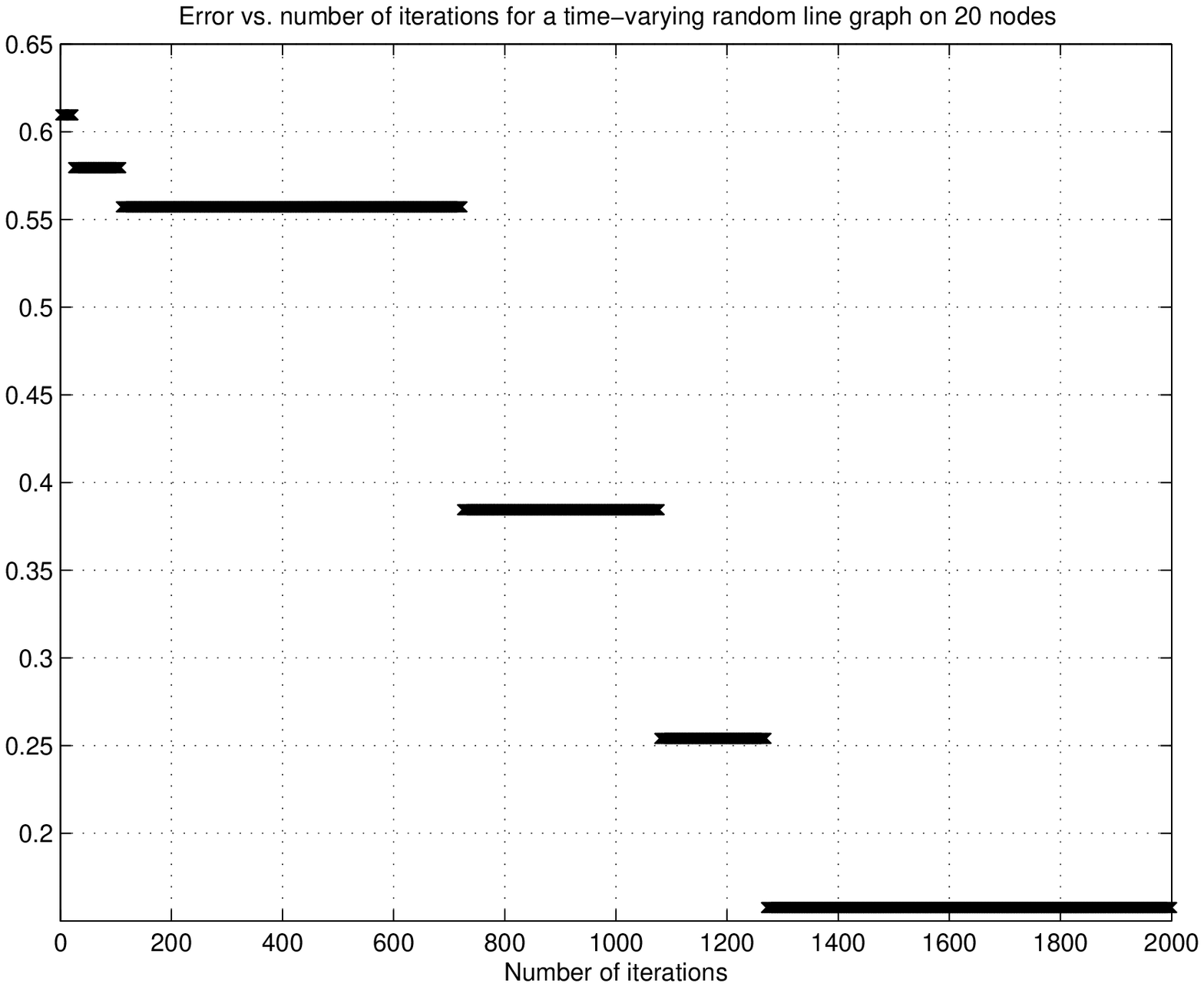} &
\includegraphics[width=3.2in]{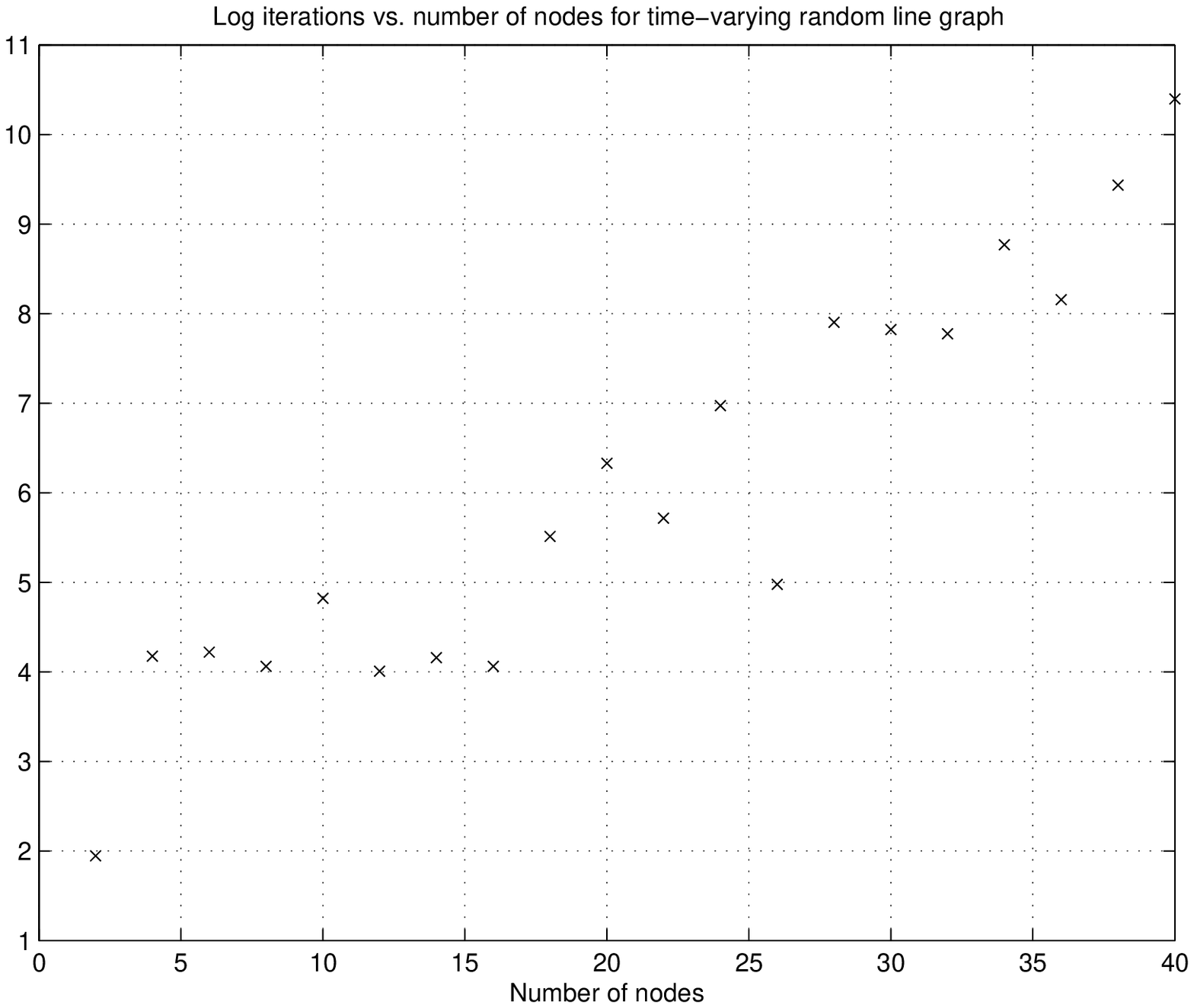} 
\end{array}$ 
\end{center} \caption{Both graphs show our protocol on a time-varying line graph which is randomly generated at each stage. The left
graph shows the decay of $\max_i |x_i(t) - \frac{1}{n} \sum_{j=1}^n x_j(0)|$ with time while the right graph shows the {\bf logarithm} of the time it takes to achieve 
$\max_i |x_i(t) - \frac{1}{n} \sum_{j=1}^n x_j(0)| \leq 0.05$.  The initial condition is random on the left and equal to 
the vector $x_1(0)=1$ and $x_k(0)=1$ for $k=2, \ldots, n$ on the right. In both cases, as before, $\alpha=0.9$.}  
\label{varying}
\end{figure}

Finally, in Figure \ref{varying}, we show our protocol on a time-varying graphs. We generate a random line at every stage by first relabeling 
the vertices uniformly at random and then connecting them in a standard line (i.e., putting an edge between $1$ and $2$, $2$ and $3$, and so forth). We see that the protocol appears to work, but the convergence
time is considerably slower than for either of our fixed graph examples. 

%

\section{Conclusions\label{conclusions}} We have provided a consensus protocol in which nodes exchange ternary messages at every step. In contrast
to the previous literature, our protocol can handle time-varying graph sequences, requires the exchange of only finitely many bits at every step, comes with a 
polynomial-time convergence rate, and does not need knowledge of either the sequence $G(t)$ or the number of agents to be implemented. 

Our paper naturally raises a number of open questions. First, our protocol uses more storage than the standard consensus algorithm due to the need
by each node $i$ to maintain the estimates $\hat{x}_{i,j, {\rm in}}, \hat{x}_{i,j, {\rm out}}$. Is it possible to avoid this without using 
``unphysical'' operations (for example,  interleaving the digits of two real numbers to form a single real number)? 

Secondly, our convergence time result was demonstrated to hold on $B$-core-connected graph sequences. However, most of the 
previous literature on consensus protocol has relied on the weaker assumption of $B$-connected 
graph sequences. An open question is to derive a similar convergence time result on the larger class of $B$-connected sequences. 

Third, while the standard consensus protocol can be said to rely on broadcasts in the sense that each neighbor of node $j$ measures or
receives the
same value $x_j(t)$ at time $t$, our protocol relies on individualized messages from every node to each of its neighbors. We wonder whether there exists a consensus
protocol with the same properties as ours (finitely many bits exchanged at each step, polynomial convergence rate on time-varying graphs, 
does not require knowledge of the graph sequence or of the number of which agents)  which relies only on broadcasts by each node at every step. 

Finally, the relation between speed, storage, and communication overhead in consensus is still poorly understood. There appear to be
some tradeoffs involved between these quantities; for example, the current paper which puts aside the assumption that nodes can transmit 
real numbers derives a slower convergence time bound in contrast to the best average consensus convergence time \cite{noot09} and has a larger amount
of storage at each node. However, whether there are indeed any formal tradeoffs between these quantities is not clear at present.

\section{Acknowledgements} The author would like to thank Zhirong Qiu for pointing out several errors in an earlier version of this
manuscript.


\begin{thebibliography}{99} 

\bibitem{bgps06} S. Boyd, A. Ghosh, B. Prabhakar, D. Shah, ``Randomized gossip algorithms,'' 
{\em IEEE Transactions on Information Theory}, vol. 52, no. 6, pp. 2508-2530, 2006.

\bibitem{bhot05} V.D. Blondel, J.M. Hendrickx, A. Olshevsky, J.N. Tsitsiklis, ``Convergence in multiagent coordination, consensus, and flocking,''
{\em Proceedings of the 44th IEEE Conference on Decision and Control}, Seville, Spain, 2005.


\bibitem{cai1} K. Cai, H. Ishii, ``Quantized consensus and averaging on gossip digraphs,'' {\em IEEE Transactions on Automatic Control}, vol. 56, no. 9, p. 2087-2100, 2011. 

\bibitem{cai2} K. Cai, H. Ishii, ``Convergence time analysis of quantized gossip consensus on digraphs,''
{\em Automatica}, vol. 48, no. 9, p. 2344-2351, 2012 

\bibitem{cbz10} R. Carli, F. Bullo and S. Zampieri, ``Quantized average consensus via dynamic coding/decoding schemes,'' 
{\em International
Journal of Nonlinear and Robust Control}, vol. 20, no. 2, pp. 156-175, 2010.

\bibitem{ccsz11} R. Carli, A. Chiuso, L. Schenato, S. Zampieri, ``Optimal synchronization for networks of noisy double integrators,'' {\em IEEE Transactions on Automatic Control}, vol. 56,
no. 5, pp. 1146�1152, 2011.  

\bibitem{cffz10} R. Carli, F. Fagnani, P. Frasca and S. Zampieri, ``Gossip consensus algorithms via quantized communication,'' {\em Automatica},
vol. 46, no. 1, pp. 70-80, 2010.

\bibitem{gcb08} C. Gao, J. Cortes, F. Bullo, ``Notes on averaging over acyclic graphs and discrete coverage control,''
{\em Automatica}, vol. 44,
no. 8, pp. 2120-2127, 2008

\bibitem{clvw05} J.S. Caughman, G. Lafferriere, J.J.P. Veerman, and A. Williams, ``Decentralized control of vehicle formations,''
{\em Systems and Control Letters}, vol. 54, no. 9, pp.
899�910, 2005. 

\bibitem{cbh09} H.L. Choi, L. Brunet, J.P. How, ``Consensus-based decentralized auctions for robust task allocation,'' 
{\em IEEE Transactions on Robotics}, vol. 25, no. 4, pp. 912-926, 2009.

\bibitem{dsw08} A. Dimakis, A.D. Sarwate, M. J. Wainwright, ``Geographic gossip: efficient averaging for sensor networks,'' 
{\em IEEE Transactions on Signal Processing}, vol. 56, no. 3, pp. 1205-1216, 2008.
 

\bibitem{fm04} J.A. Fax and R.M. Murray, ``Information flow and cooperative control of
vehicle formations,'' {\em IEEE Transactions on Automatic Control}, vol. 49, no. 9, pp.
1465�1476, 2004. 

\bibitem{frasca} P. Frasca, ``Continuous-time quantized consensus: convergence of Krasovskii solutions,''
{\em Systems and Control Letters}, vol. 61, no. 2, pp. 273-278, 2012 

\bibitem{fcfz09} P. Frasca, R. Carli, F. Fagnani and S. Zampieri, ``Average consensus on networks with quantized communication,''
{\em International Journal of Nonlinear and Robust Control}, vol. 19, no. 16, pp. 1787-1816, 2009.

\bibitem{dim} M. Guo, D. V. Dimarogonas, ``Consensus with quantized relative state measurements,'' 
{\em Automatica}, vol. 49, no. 8, pp. 2531-2537, 2013.



\bibitem{JLM03} A.\  Jadbabaie, J.\  Lin, and A.\ S.\  Morse, ``Coordination of groups
of mobile autonomous agents using nearest neighbor rules,'' {\it
IEEE Transactions on Automatic Control}, vol.\ 48, no.\  3, pp.\
988-1001, 2003.


\bibitem{kbs07} A. Kashyap, T. Basar and R. Srikant, ``Quantized consensus,'' {\em Automatica}, vol. 43, no. 7, pp. 1192-1203, 2007.


\bibitem{lby03} J.R. Lawton, R.W. Beard, and B. Young, ``A decentralized approach to formation maneuvers,''
{\em  IEEE Transactions on Robotics and  Automation}, vol. 19, no. 6, pp. 933�941,
2003.

\bibitem{lfxz11} T. Li, M. Fu, L. Xie and J. F. Zhang, ``Distributed consensus with limited communication data rate,''
{\em IEEE Transactions
on Automatic Control}, vol. 56, no. 2, pp. 279-292, 2011

\bibitem{xie2} S. Liu, T. Li, L. Xie, M. Fu, J. F. Zhang, ``Continuous-time and sampled-data based
average consensus with logarithmic quantizers,'' preprint. 

\bibitem{xiesiam} S. Liu, T. Li, L. Xie, 	``Distributed consensus for multi-agent systems with communication 
delays and limited data rate,'' {\em SIAM Journal on Control and Optimization}, vol. 49, no. 6, pp. 2239-2262, 2011.

\bibitem{lm12} J. Lavaei, R.M. Murray, ``Quantized consensus by means of gossip algorithm,'' 
{\em IEEE Transactions on Automatic Control}, vol. 57, no. 1, pp. 19-32, 2012. 

\bibitem{lr06} Q. Li, D. Rus, ``Global clock synchronization in sensor networks,'' {\em IEEE Transactions on Computers}, 
vol 55, no. 2, pp. 214-226, 2006. 

\bibitem{mkb04} S. Martinez, T. Karatas, F. Bullo, ``Coverage control for mobile sensing networks,'' 
{\em IEEE Transactions on Robotics and Automation}, vol. 20, no. 2, pp. 243-255, 2004.

\bibitem{mkb07} S. Martinez, J. Cortes, F. Bullo, ``Motion coordination with distributed information,'' 
{\em IEEE Transactions on Control Systems Technology}, vol. 27, no. 4, pp. 75-88, 2007.


\bibitem{M05} L. Moreau, ``Stability of multi-agent systems with time-dependent communication links,''
{\em  IEEE Transactions on Automatic Control}, vol. 50, pp. 169-182, 2005.


\bibitem{noot09} A. Nedic, A. Olshevsky, A. Ozdaglar and J. N. Tsitsiklis, ``On distributed averaging algorithms and quantization effects,''
{\em IEEE Transactions on Automatic Control}, vol. 54, no. 11, pp. 2506-2517, 2009.

\bibitem{OSM04} R. Olfati-Saber, R.M. Murray, ``Consensus problems in networks of agents with switching topology and time-delays,''
{\em IEEE Transactions on Automatic Control}, vol. 49, no. 3, pp. 1250-1533, 2004. 

\bibitem{finite} F. Pasqualetti, D.  Borra, F.  Bullo, ``Consensus networks over finite fields,'' preprint, 
available at \url{http://arxiv.org/abs/1301.4587}.

\bibitem{penrose} M. Penrose, {\em Random Geometric Graphs}, Oxford University Press, 2003. 

\bibitem{princeton} S. Shang, P. W. Cuff, P. Hui, S. R. Kulkarni, ``An upper bound on the convergence time for 
quantized consensus,'' preprint, available at \url{http://arxiv.org/abs/1208.0788}.

\bibitem{sf11} L. Schenato and F. Fiorentin, ``Average timesynch: a consensus-based protocol for clock synchronization in wireless sensor
networks,'' {\em Automatica}, vol. 47, no. 9, pp. 1878�1886, 2011.


\bibitem{TBA86} J. N. Tsitsiklis, D. P. Bertsekas, and M. Athans,
``Distributed asynchronous deterministic and stochastic gradient
optimization algorithms,'' {\em IEEE Transactions on Automatic
Control,} vol. 31, no. 9, 1986, pp. 803-812.


\bibitem{xbk07} L. Xiao, S. Boyd, S.J. Kim, ``Distributed average consensus with least-mean square
deviation,'' {\em Journal of Parallel and Distributed Computing}, vol. 67, no. 1, pp. 33-46, 2007.


\end{thebibliography}
\end{document}